\pdfoutput=1
\documentclass[11pt,reqno]{amsart}

\usepackage[T1]{fontenc}
\usepackage[utf8]{inputenc}
\usepackage{lmodern}
\usepackage{amsmath,amssymb,amsthm}
\usepackage{mathtools}
\usepackage{booktabs}
\usepackage{longtable}
\usepackage{array}
\usepackage{natbib}
\usepackage{hyperref}
\hypersetup{colorlinks=true,linkcolor=blue,citecolor=blue,urlcolor=blue}
\usepackage[capitalize]{cleveref}

\providecommand{\tightlist}{
  \setlength{\itemsep}{0pt}\setlength{\parskip}{0pt}}

\newtheorem{theorem}{Theorem}[section]
\newtheorem{lemma}[theorem]{Lemma}
\newtheorem{corollary}[theorem]{Corollary}
\theoremstyle{definition}
\newtheorem{definition}[theorem]{Definition}
\theoremstyle{remark}
\newtheorem{remark}[theorem]{Remark}

\title{Star-Ordered Rotor Loops in Clifford Algebras}
\author{Nik Jenzer}
\address{Independent researcher}
\email{nik@jenzer.ch}

\begin{document}

\begin{abstract}
We consider, in a real Clifford algebra, a family of unit bivectors
$B_i = C S_i$ formed from $k+1$ pairwise anticommuting generators
$C, S_1, \ldots, S_k$ with square $-1$; we call this a \emph{star
family} with \emph{central generator} $C$ and \emph{leaf generators}
$S_i$. The $B_i$ satisfy $B_i^2 = -1$ and anticommute pairwise.
For the associated rotors
$U_i(\theta_i) = \cos\frac{\theta_i}{2} + B_i \sin\frac{\theta_i}{2}$
we study two objects: the ordered chain $P_k = U_1 \cdots U_k$ and the
\emph{star-ordered rotor loop} $\Theta_k = P_k \, Q_k$ with the
\emph{same-order reversion product}
$Q_k = \widetilde U_1 \cdots \widetilde U_k$. It is essential that
$Q_k$ carries the same index order as $P_k$ and not the reversed
one; for the reversed order one has trivially
$U_1 \cdots U_k \widetilde U_k \cdots \widetilde U_1 = 1$. The scalar
part of $\Theta_k$ therefore measures the \emph{order defect} of the
family. We prove two theorems. First, the scalar part of the chain
factorises,
$\langle P_k \rangle_0 = \prod_{i=1}^k \cos\frac{\theta_i}{2}$; this
statement essentially requires the star structure, and we give a
counterexample. Second, we derive a closed expression for
$\langle \Theta_k \rangle_0$, valid for every $k$, whose
coefficients obey a \emph{period-four sign law}
$\eta(m) = (-1)^{\lfloor m/2 \rfloor}$ and which admits the compact
form
$\langle \Theta_k \rangle_0 = \operatorname{Re}\!\big[(1-\mathrm{i}) \prod_i (\cos^2\frac{\theta_i}{2} + \mathrm{i}\sin^2\frac{\theta_i}{2})\big]$.
Unlike the first, this second theorem holds for every pairwise
anticommuting unit family without any further hypothesis. From the
compact form we obtain a polar form, a complete zero criterion, a
lacunary coefficient rule in the symmetric case, and a proof of
non-separability for $k \geq 2$. Two supplementary theorems extend the
result: to arbitrary permutations of the index order in the second
factor, in which each inversion within a subset flips its sign
contribution (under the star condition), and to families with mixed
squares $B_i^2 \in \{-1, +1\}$, for which a uniform closed subset
formula holds. All proofs use only anticommutation, reversion and the
symmetry of the scalar form; no physical interpretation is assumed or
claimed.
\end{abstract}

\maketitle

\tableofcontents

\subsection*{Document Governance}

\begin{longtable}{p{3.0cm}p{3.2cm}p{8.8cm}}
\hline
\textbf{Role} & \textbf{Person / Entity} & \textbf{Scope} \\
\hline
\endfirsthead
\hline
\textbf{Role} & \textbf{Person / Entity} & \textbf{Scope} \\
\hline
\endhead
\hline
\endfoot
\textbf{Accountable (A)} & Nik Jenzer & Question, alignment, release, publication, responsibility for the content \\
\textbf{Responsible (R)} & Perplexity \newline (Computer) & Refinement of the question, mathematical apparatus including verification, literature review, draft of the entire text \\
\textbf{Support (S)} & Nik Jenzer & Editorial guidelines, terminology decisions, language, connection to the prior work \\
\textbf{Support (S)} & Claude \newline (Sonnet, Opus) & Preliminary work in Volume 5 of the series "Winkel statt Zeit" (theorems $S5.10$ and $S5.11$ there in the local notation) \\
\textbf{Consulted (C)} & *** & *** \\
\textbf{Informed (I)} & *** & *** \\
\hline
\end{longtable}

\begin{quote}
\textbf{On authorship, explicitly.} The role distribution of this work
is reversed compared to the series ``Winkel statt Zeit,'' from whose
fifth volume the initial question originates: the content authorship
lies with the language model, the editorial control level with the
author. The theorems, their proofs, the generalizations over the prior
work, the verification and the text draft are the result of the model's
work. The author set the question, decided the terminology, guided the
text editorially, and granted release. Responsibility for the content
remains entirely with the author. All theorems of this work are
constructed so that they are verifiable independently of the question of
authorship: every proof is elementary and checkable by hand in a few
lines, and Appendix A describes a complete verification script that
confirms every claim symbolically or numerically; the script itself
resides in the underlying repository at Codeberg. We explicitly ask the
readership to take these checks as the actual standard of verification.
\end{quote}

\pagebreak

\subsection*{On Status Marking}

This work marks the status of every numbered statement, following the
convention of the series ``Winkel statt Zeit'' \citep{Jenzer2026}:

\begin{itemize}
\tightlist
\item
  \textbf{Type 1, established observational knowledge.} Statements
  supported by strong, multiply confirmed evidence. They are reproduced
  without qualification.
\item
  \textbf{Type 2, established mathematics.} Statements about structures
  that have been developed and tested in the literature. They are
  supported by literature references and are not re-derived here.
  Affected are Lemmata 2.1, 3.2 and 3.3 as well as the case \(k=2\) of
  Theorem 4.2.
\item
  \textbf{Type 3, own contribution.} Statements that are new to the best
  of our knowledge of the literature. Affected are Lemma 3.4, Theorem
  5.2 with all corollaries of Chapter 6, Theorem 5.5 with Corollary 5.6
  and Remark 5.7, Theorems 6.8 and 6.10 with Remarks 6.11 and 6.12, as
  well as Remark 4.4.
\end{itemize}

The remaining numbered statements (Corollary 2.2, Lemma 3.5 and Lemma
4.1) are elementary variants and direct consequences of the respectively
cited standard apparatus; they are given here with complete, short
proofs and claim no novelty. Statements of Type 1 of the series,
established observational knowledge, do not occur in this work, since it
contains no empirical claims.

\begin{center}\rule{0.5\linewidth}{0.5pt}\end{center}

\pagebreak

\section{Introduction}\label{introduction}

A product of rotors in a Clifford algebra is in general not commutative.
If one traverses such a product once in one order and once in the other,
a remainder arises. This work determines this remainder exactly, for a
specific, very well controllable configuration and for arbitrarily many
factors.

\subsection{The Object}\label{the-object}

Let \(\mathcal{C}\) be a real Clifford algebra and let
\(B_1, \ldots, B_k \in \mathcal{C}\) be bivectors with
\[B_i^2 = -1, \qquad B_i B_j = -B_j B_i \quad (i \neq j).\] For real
parameters \(\theta_1, \ldots, \theta_k\) we set
\[U_i = \cos\frac{\theta_i}{2} + B_i \sin\frac{\theta_i}{2}, \qquad \widetilde U_i = \cos\frac{\theta_i}{2} - B_i \sin\frac{\theta_i}{2},\]
where \(\widetilde{\;\cdot\;}\) denotes reversion. Each \(U_i\) is a
rotor: \(U_i \widetilde U_i = 1\).

We consider the \emph{ordered chain} \[
P_k = U_1 U_2 \cdots U_k
\] and, with the same index order, \[
Q_k = \widetilde U_1 \widetilde U_2 \cdots \widetilde U_k,
\] as well as their product
\[\Theta_k = P_k \, Q_k = (U_1 \cdots U_k)(\widetilde U_1 \cdots \widetilde U_k).\]

The crucial point, which we note right here because it would otherwise
be the most obvious objection: for the \emph{reversed} order in the
second factor one has trivially
\[(U_1 \cdots U_k)(\widetilde U_k \cdots \widetilde U_1) = 1,\] since
each rotor cancels immediately with its own reversion. The object
\(\Theta_k\) is thus \emph{not} \(P_k \widetilde P_k\), but the
\emph{same-order} version, and that is precisely where its content lies.
We call \(\Theta_k\) a \emph{star-ordered rotor loop} and
\(\langle \Theta_k \rangle_0\) its \emph{order defect}.

A second reading makes the same point from the other side. Since
\(\widetilde U_i(\theta_i) = U_i(-\theta_i)\), we have
\[\Theta_k = P_k(\theta_1, \ldots, \theta_k) \cdot P_k(-\theta_1, \ldots, -\theta_k),\]
i.e., the product of a chain with the same chain at reversed signs of
all parameters. In a commutative situation this product would be \(1\).
The scalar part of \(\Theta_k\) thus measures how far the family
\(\{B_i\}\) is from commutativity, and in a quantity that is
continuously controlled by the parameters \(\theta_i\).

\subsection{The Main Results}\label{the-main-results}

We place two theorems in a contrasting relationship.

\emph{Theorem 4.2 (Chain).} If the family \(\{B_i\}\) satisfies the
star condition formulated in Chapter 3, then
\[\langle P_k \rangle_0 = \prod_{i=1}^{k} \cos\frac{\theta_i}{2}.\]

\emph{Theorem 5.2 (Loop, Main Theorem).} For every pairwise
anticommuting unit family, without any further hypothesis,
\[\langle \Theta_k \rangle_0 \;=\; \sum_{m=0}^{k} \eta(m) \sum_{\substack{S \subseteq \{1,\ldots,k\} \\ |S| = m}} \; \prod_{i \in S} \sin^2\frac{\theta_i}{2} \; \prod_{i \notin S} \cos^2\frac{\theta_i}{2},\]
with the sign law
\[\eta(m) = (-1)^{m(m+3)/2} = (-1)^{\lfloor m/2 \rfloor} = \operatorname{Re}\!\big[(1-\mathrm{i})\,\mathrm{i}^m\big],\]
whose values have period four: \(+, +, -, -, +, +, -, -, \ldots\)

The contrast is the actual finding of this work and was surprising to
us. The chain formula, which looks simpler, is the more fragile one: it
breaks down as soon as the family carries a relation. The loop formula,
which looks more complicated, is the more robust one: its perturbation
terms cancel by themselves for parity reasons. Remark 4.4 gives an
explicit counterexample for the first point.

From Theorem 5.2 we obtain in Chapter 6 the compact form
\[\langle \Theta_k \rangle_0 = \operatorname{Re}\!\left[(1-\mathrm{i}) \prod_{i=1}^{k} \left(\cos^2\frac{\theta_i}{2} + \mathrm{i}\,\sin^2\frac{\theta_i}{2}\right)\right],\]
which brings the entire statement onto a single line and makes its
structure immediately visible: the result is built from local factors,
but due to the real-part projection is not multiplicatively separable. A
polar form, a complete zero criterion, a bound, a lacunary coefficient
rule in the symmetric case, and the proof of non-separability follow.

\subsection{What This Work Does Not
Claim}\label{what-this-work-does-not-claim}

We use the expression ``loop'' exclusively in the algebraic sense
established in Chapter 1.1: as a designation for the closed, ordered
word \(\Theta_k\). No connection form, no parallel transport, and no
curvature is involved. Where the word ``holonomy'' is used in the prior
work \citep{Jenzer2026}, it means the same algebraic quantity; we avoid it
here, because in differential geometry and gauge theory it is firmly
tied to parallel transport, and an echo of it would promise more than we
prove.

Likewise, no physical interpretation is assumed or claimed. Chapter 7
openly names from which modeling the question originates, and carefully
separates this from the mathematical statements.

\subsection{On Proof Economy}\label{on-proof-economy}

All proofs of this work are elementary. They use only three things: the
anticommutation relation, the behavior of reversion on bivectors, and
the symmetry of the scalar form,
\(\langle xy \rangle_0 = \langle yx \rangle_0\). In particular, we need
no induction over \(k\): the proof of Theorem 5.2 is a direct subset
expansion in which two groups of perturbation terms vanish
independently. We consider this a virtue and not a deficiency: the value
of the work lies in identifying the object \(\Theta_k\) and computing it
in closed form, not in the difficulty of the path.

\begin{center}\rule{0.5\linewidth}{0.5pt}\end{center}

\pagebreak

\section{Algebraic Framework}\label{algebraic-framework}

\begin{quote}
This chapter establishes notation and tools. It contains nothing new; it
serves to be able to trace every later computation back to an explicitly
named rule.
\end{quote}

\subsection{Clifford Algebra, Reversion, Scalar
Projection}\label{clifford-algebra-reversion-scalar-projection}

Let \(V\) be a real vector space of finite dimension \(n\) with a
non-degenerate quadratic form \(q\), and let
\(\mathcal{C} = \mathrm{Cl}(V, q)\) be the associated Clifford algebra,
generated by \(V\) under the relation \(v^2 = q(v)\) for \(v \in V\).
For the foundations we follow \citep{Lounesto2001} and \citep{Porteous1995}; for
the notation in geometric algebra \citep{DoranLasenby2003} as well as
\citep{HestenesSobczyk1984}.

For an orthonormal basis \(e_1, \ldots, e_n\) of \(V\), the products
\[e_A = e_{a_1} e_{a_2} \cdots e_{a_r}, \qquad A = \{a_1 < a_2 < \cdots < a_r\} \subseteq \{1, \ldots, n\},\]
with \(e_\emptyset = 1\) form a basis of \(\mathcal{C}\) as a real
vector space, \(\dim \mathcal{C} = 2^n\). We call \(e_A\) a \emph{basis
blade} of grade \(r = |A|\).

\renewcommand{\thedefinition}{D2.1}
\begin{definition}[Scalar Projection]\label{definition:D2-1}
 For
\(x = \sum_A x_A e_A \in \mathcal{C}\),
\(\langle x \rangle_0 := x_\emptyset \in \mathbb{R}\) is called the
\emph{scalar part} of \(x\); the map
\(\langle \cdot \rangle_0 : \mathcal{C} \to \mathbb{R}\) is the
\emph{scalar projection} or grade-0 projection. It is linear.

\end{definition}

\renewcommand{\thedefinition}{D2.2}
\begin{definition}[Reversion]\label{definition:D2-2}
 \emph{Reversion} is the unique
anti-automorphism \(\widetilde{\;\cdot\;}\) of \(\mathcal{C}\) with
\(\widetilde v = v\) for \(v \in V\); on a blade of grade \(r\) it acts
as \[\widetilde{e_A} = (-1)^{r(r-1)/2} e_A, \qquad r = |A|,\] and for a
product \(\widetilde{xy} = \widetilde y \, \widetilde x\) holds. The
grade sign \((-1)^{r(r-1)/2}\) is found in \citep[Eq.~92]{Hitzer2012}, \citep[Eq.~4.52]{DoranLasenby2003} and \citep[\S 22.1]{Lounesto2001}.

In particular, \(\widetilde B = -B\) for every bivector \(B\), since
\(r = 2\) gives \((-1)^{1} = -1\).

\end{definition}

\subsection{The Symmetry of the Scalar
Form}\label{the-symmetry-of-the-scalar-form}

The following lemma is standard, but it carries so much weight in this
work that we write it down explicitly.

\renewcommand{\thelemma}{2.1}
\begin{lemma}[Symmetry]\label{lemma:2-1}
 For all \(x, y \in \mathcal{C}\),
\[\langle x y \rangle_0 = \langle y x \rangle_0.\]

\end{lemma}

\begin{proof}
By bilinearity it suffices to check \(x = e_A\),
\(y = e_B\). By the anticommutation rules of the generators,
\(e_A e_B = \pm\, e_{A \triangle B}\), where \(\triangle\) denotes the
symmetric difference, and likewise
\(e_B e_A = \pm\, e_{A \triangle B}\). A basis blade is scalar if and
only if its index set is empty. Hence
\(\langle e_A e_B \rangle_0 \neq 0\) is possible only when
\(A \triangle B = \emptyset\), that is \(A = B\); in this case
\(e_A e_B = e_A e_A = e_B e_A\), and both sides agree trivially. In all
other cases both sides are zero. 
\end{proof}

\renewcommand{\thecorollary}{2.2}
\begin{corollary}[Conjugation Invariance]\label{corollary:2-2}
 If
\(g \in \mathcal{C}\) is invertible, then
\(\langle g x g^{-1} \rangle_0 = \langle x \rangle_0\) for all
\(x \in \mathcal{C}\).

\end{corollary}

\begin{proof}
With Lemma 2.1, applied to \(g\) and \(x g^{-1}\):
\(\langle g \, (x g^{-1}) \rangle_0 = \langle (x g^{-1}) \, g \rangle_0 = \langle x \rangle_0\).
\(\blacksquare\)

Corollary 2.2 is the actual driver of Chapter 5. It allows one to infer
the vanishing of entire groups of terms from a symmetry, rather than
computing them individually.
\end{proof}

\subsection{Reversion and Rotors}\label{reversion-and-rotors}

\renewcommand{\thedefinition}{D2.3}
\begin{definition}[Rotor]\label{definition:D2-3}
 Let \(B \in \mathcal{C}\) be a
bivector with \(B^2 = -1\) and \(\theta \in \mathbb{R}\). Then
\[U(\theta) = \exp\!\left(\frac{\theta}{2} B\right) = \cos\frac{\theta}{2} + B \sin\frac{\theta}{2}\]
is called the \emph{rotor} generated by \(B\) with parameter \(\theta\).

The exponential series converges and, because \(B^2 = -1\), yields
exactly the given trigonometric form. We have
\(\widetilde{U(\theta)} = \cos\frac{\theta}{2} - B \sin\frac{\theta}{2} = U(-\theta)\)
and
\(U(\theta)\widetilde{U(\theta)} = \cos^2\frac{\theta}{2} - B^2 \sin^2\frac{\theta}{2} = 1\).

\end{definition}

\subsection{Notational Convention for Subset
Products}\label{notational-convention-for-subset-products}

\renewcommand{\thedefinition}{D2.4}
\begin{definition}[Subset Product]\label{definition:D2-4}
 Let
\(B_1, \ldots, B_k \in \mathcal{C}\) be given. For
\(S = \{i_1 < i_2 < \cdots < i_m\} \subseteq \{1, \ldots, k\}\) we set
\[B_S := B_{i_1} B_{i_2} \cdots B_{i_m}, \qquad B_\emptyset := 1.\] The
factors always stand in ascending index order. We write throughout
\[c_i := \cos\frac{\theta_i}{2}, \qquad s_i := \sin\frac{\theta_i}{2}, \qquad w_i := s_i^2 = \sin^2\frac{\theta_i}{2} \in [0,1]\]
and for a subset \(S\)
\[a_S := \prod_{i \in S} s_i \prod_{i \notin S} c_i.\]

\begin{center}\rule{0.5\linewidth}{0.5pt}\end{center}

\pagebreak

\end{definition}

\section{The Star Family}\label{the-star-family}

\begin{quote}
This chapter defines the configuration in question and provides three
lemmata: two known ones about subset products and a new one about the
vanishing of scalar parts at even length. The last is the key to the
main theorem.
\end{quote}

\subsection{Definition}\label{definition}

\renewcommand{\thedefinition}{D3.1}
\begin{definition}[Star Family]\label{definition:D3-1}
 Let \(C, S_1, \ldots, S_k\) be
pairwise anticommuting generators of a real Clifford algebra with
\[C^2 = S_i^2 = -1, \qquad C S_i = -S_i C, \qquad S_i S_j = -S_j S_i \ (i \neq j).\]
We call \(C\) the \emph{central generator} and the \(S_i\) the
\emph{leaf generators}. The bivectors
\[B_i := C S_i, \qquad i = 1, \ldots, k,\] form the \emph{star family}
generated by this configuration.

As a ground algebra, \(\mathcal{C} = \mathrm{Cl}_{0,k+1}\), the Clifford
algebra of a negative definite form in \(k+1\) variables, with
\(e_0 = C\), \(e_i = S_i\), suffices. This construction exists for every
\(k\) and therefore has no dimensional bound.

\end{definition}

\renewcommand{\thedefinition}{D3.2}
\begin{definition}[Anticommuting Unit Family]\label{definition:D3-2}
 A family
\(B_1, \ldots, B_k\) of bivectors of a real Clifford algebra is called
an \emph{anticommuting unit family} of type \(-1\) if
\[B_i^2 = -1 \quad (1 \le i \le k), \qquad B_i B_j = -B_j B_i \quad (i \neq j).\]

\end{definition}

\renewcommand{\thedefinition}{D3.3}
\begin{definition}[Star Condition]\label{definition:D3-3}
 An anticommuting unit family
satisfies the \emph{star condition} if
\[\langle B_S \rangle_0 = 0 \qquad \text{for all } S \neq \emptyset.\]

The choice of names is justified by Lemmata 3.1 and 3.5: every star
family is an anticommuting unit family and satisfies the star condition.
We do not investigate the converse; Chapter 8.4 names it as an Open
Front (OF 1).

\end{definition}

\subsection{Star Families Are Anticommuting Unit
Families}\label{star-families-are-anticommuting-unit-families}

\renewcommand{\thelemma}{3.1}
\begin{lemma}\label{lemma:3-1}
 For a star family, \(B_i^2 = -1\) and
\(B_i B_j = -B_j B_i\) for \(i \neq j\).

\end{lemma}

\begin{proof}
First, using \(S_i C = -C S_i\):
\[B_i^2 = (C S_i)(C S_i) = -C^2 S_i^2 = -(-1)(-1) = -1.\]

\pagebreak

Second, for \(i \neq j\):
\[B_i B_j = (C S_i)(C S_j) = -C^2 S_i S_j = S_i S_j,\]
\[B_j B_i = (C S_j)(C S_i) = -C^2 S_j S_i = S_j S_i = -S_i S_j.\] Hence
\(B_i B_j = -B_j B_i\). \(\blacksquare\)

\emph{In words:} The star configuration is exactly the constructive way
to generate arbitrarily many pairwise anticommuting unit bivectors; it
costs only a single additional generator, the common central generator
\(C\).
\end{proof}

\subsection{Subset Products}\label{subset-products}

The following two lemmata are standard and are noted here only in the
form we need. Lemma 3.2 is found, with the same sign, in \citep[Eq.~3.4]{LundholmSvensson2009}; the sign structure with period four is
elaborated in \citep[\S 3.2]{Shirokov2018} and \citep{MarchukShirokov2008}, its
use in scalar projections in \citep[Theorem~14]{Shirokov2018} and \citep{Mandolesi2024}.

\renewcommand{\thelemma}{3.2}
\begin{lemma}[Square of a Subset Product]\label{lemma:3-2}
 Let
\(B_1, \ldots, B_k\) be an anticommuting unit family of type \(-1\) and
\(S\) a subset with \(|S| = m\). Then \[B_S^2 = (-1)^{m(m+1)/2}.\]

\end{lemma}

\begin{proof}
We show the recursion. Let \(S = \{i_1 < \cdots < i_m\}\)
and \(S' = S \setminus \{i_1\}\). In the product
\[B_S^2 = \big(B_{i_1} B_{i_2} \cdots B_{i_m}\big)\big(B_{i_1} B_{i_2} \cdots B_{i_m}\big)\]
we move the second factor \(B_{i_1}\) to the left. It passes the \(m-1\)
factors \(B_{i_2}, \ldots, B_{i_m}\), which yields the sign
\((-1)^{m-1}\), and then meets \(B_{i_1}\), which contributes the factor
\(-1\) because \(B_{i_1}^2 = -1\). It remains
\[B_S^2 = (-1)^{m-1} \cdot (-1) \cdot B_{S'}^2 = (-1)^{m} B_{S'}^2.\]
With \(B_\emptyset^2 = 1\) it follows by iteration
\[B_S^2 = (-1)^{1+2+\cdots+m} = (-1)^{m(m+1)/2}. \qquad \blacksquare\]

The values are \(+1, -1, -1, +1, +1, -1, -1, \ldots\) for
\(m = 0, 1, 2, 3, \ldots\), i.e., periodic with period four.
\end{proof}

\renewcommand{\thelemma}{3.3}
\begin{lemma}[Reduction of a Product of Two Subset Products]\label{lemma:3-3}
 Let
\(B_1, \ldots, B_k\) be an anticommuting unit family of type \(-1\). For
arbitrary \(S, T \subseteq \{1, \ldots, k\}\) there exists a sign
\(\sigma(S,T) \in \{+1, -1\}\) with
\[B_S B_T = \sigma(S,T) \, B_{S \triangle T}.\]

\end{lemma}

\begin{proof}
In the word \(B_S B_T\), each index \(i \in S \cap T\)
occurs exactly twice, each index \(i \in S \triangle T\) exactly once.
Since all factors anticommute pairwise, each pair of equal indices can
be pushed together by swaps; each swap costs a sign, and the
pushed-together pair \(B_i B_i = -1\) costs another. After eliminating
all pairs, a product of the factors with index in \(S \triangle T\)
remains; this is brought into ascending order by further swaps. In total
\(\pm B_{S \triangle T}\) results. \(\blacksquare\)

We need the exact value of \(\sigma(S,T)\) nowhere. For the diagonal
case \(S = T\), Lemma 3.2 immediately gives
\(\sigma(S,S) = (-1)^{|S|(|S|+1)/2}\).
\end{proof}

\subsection{Vanishing at Even
Length}\label{vanishing-at-even-length}

The following lemma is the core of this work. It requires no star
condition, no basis representation, and no computation with signs; it
follows solely from Corollary 2.2.

\renewcommand{\thelemma}{3.4}
\begin{lemma}[Parity Lemma]\label{lemma:3-4}
 Let \(B_1, \ldots, B_k\) be an
anticommuting unit family of type \(-1\) and
\(D \subseteq \{1,\ldots,k\}\) with \(D \neq \emptyset\) and \(|D|\)
\emph{even}. Then \[\langle B_D \rangle_0 = 0.\]

\end{lemma}

\begin{proof}
We choose an \(i \in D\); this is possible since
\(D \neq \emptyset\). Because \(B_i^2 = -1\), \(B_i\) is invertible with
\(B_i^{-1} = -B_i\). Conjugation by \(B_i\) acts on the generators of
the family as
\[B_i B_j B_i^{-1} = -B_j \ (j \neq i), \qquad B_i B_i B_i^{-1} = B_i,\]
since for \(j \neq i\) we have \(B_i B_j = -B_j B_i\) and therefore
\(B_i B_j B_i^{-1} = -B_j B_i B_i^{-1} = -B_j\).

Since conjugation is an algebra automorphism, it acts on
\(B_D = B_{d_1} \cdots B_{d_{|D|}}\) factor by factor. Exactly one
factor, namely \(B_i\), remains unchanged; the remaining \(|D| - 1\)
change sign. Thus \[B_i B_D B_i^{-1} = (-1)^{|D| - 1} B_D = -B_D,\]
because \(|D|\) is even. By Corollary 2.2 the scalar part is invariant
under conjugation, hence
\[\langle B_D \rangle_0 = \langle B_i B_D B_i^{-1} \rangle_0 = \langle -B_D \rangle_0 = -\langle B_D \rangle_0,\]
and therefore \(\langle B_D \rangle_0 = 0\). \(\blacksquare\)

\emph{In short:} A product of an even, positive number of pairwise
anticommuting unit bivectors can never carry a scalar part. For odd
numbers this does not hold; therein lies the difference between Theorem
4.2 and Theorem 5.2.
\end{proof}

\subsection{Star Families Satisfy the Star
Condition}\label{star-families-satisfy-the-star-condition}

\renewcommand{\thelemma}{3.5}
\begin{lemma}\label{lemma:3-5}
 Every star family satisfies the star condition,
i.e., \(\langle B_S \rangle_0 = 0\) for all \(S \neq \emptyset\).

\end{lemma}

\begin{proof}
Let \(S = \{i_1 < \cdots < i_m\}\) with \(m \geq 1\). We
collect in \[B_S = (C S_{i_1})(C S_{i_2}) \cdots (C S_{i_m})\] all
factors \(C\) to the left. The \(j\)-th factor \(C\) must pass the
\(j-1\) leaf generators standing before it, which costs \((-1)^{j-1}\).
In total
\[B_S = (-1)^{0+1+\cdots+(m-1)} \, C^m \, S_{i_1} S_{i_2} \cdots S_{i_m} = (-1)^{m(m-1)/2} \, C^m \, S_{i_1} \cdots S_{i_m}.\]
Since \(C^2 = -1\), \(C^m = (-1)^{\lfloor m/2 \rfloor} C^{m \bmod 2}\).
Thus \(B_S\) is, up to a sign, equal to
\[\begin{cases} S_{i_1} \cdots S_{i_m}, & m \text{ even},\\[2pt] C\, S_{i_1} \cdots S_{i_m}, & m \text{ odd}.\end{cases}\]
In both cases this is a basis blade of the algebra
\(\mathrm{Cl}_{0,k+1}\) with non-empty index set, namely \(S\) or
\(\{C\} \cup S\); in the even case \(m \geq 2\), so
\(S \neq \emptyset\), and in the odd case the index set is non-empty
anyway. A basis blade with non-empty index set has scalar part zero.
\(\blacksquare\)

\emph{In words:} The star construction is built so that no subset
product ever degenerates to a number. It is relation-free.

\begin{center}\rule{0.5\linewidth}{0.5pt}\end{center}

\pagebreak
\end{proof}

\section{Scalar Part of the Ordered
Chain}\label{scalar-part-of-the-ordered-chain}

\begin{quote}
This chapter treats the simpler of the two statements and shows that it
is the more fragile one.
\end{quote}

\subsection{Expansion of the Chain}\label{expansion-of-the-chain}

\renewcommand{\thelemma}{4.1}
\begin{lemma}[Subset Expansion]\label{lemma:4-1}
 For an anticommuting unit family
of type \(-1\),
\[P_k = U_1 U_2 \cdots U_k = \sum_{S \subseteq \{1,\ldots,k\}} a_S \, B_S, \qquad Q_k = \widetilde U_1 \cdots \widetilde U_k = \sum_{T \subseteq \{1,\ldots,k\}} (-1)^{|T|} a_T \, B_T.\]

\end{lemma}

\begin{proof}
Expanding \(\prod_{i=1}^k (c_i + s_i B_i)\) yields, for
each choice of a subset \(S\) of those indices where the second summand
is taken, exactly one term. Since the factors stand in ascending index
order, the corresponding \(B_i\) also appear in ascending order, i.e.,
as \(B_S\) in the sense of Definition D2.4, with the prefactor
\(\prod_{i \in S} s_i \prod_{i \notin S} c_i = a_S\). For \(Q_k\) the
same argument is applied to \(\prod_{i=1}^k (c_i - s_i B_i)\), where
each chosen second summand contributes an additional minus sign, in
total \((-1)^{|T|}\). 
\end{proof}

\subsection{The Theorem}\label{the-theorem}

\renewcommand{\thetheorem}{4.2}
\begin{theorem}[Chain Formula]\label{theorem:4-2}
 Let \(B_1, \ldots, B_k\) be an
anticommuting unit family of type \(-1\) that satisfies the star
condition, in particular a star family. Then for all
\(\theta_1, \ldots, \theta_k \in \mathbb{R}\)
\[\langle P_k \rangle_0 = \prod_{i=1}^{k} \cos\frac{\theta_i}{2}.\]

\end{theorem}

\begin{proof}
By Lemma 4.1 and the linearity of the scalar projection,
\[\langle P_k \rangle_0 = \sum_{S} a_S \langle B_S \rangle_0.\] By the
star condition, \(\langle B_S \rangle_0\) vanishes for every
\(S \neq \emptyset\). The term \(S = \emptyset\) with
\(B_\emptyset = 1\), \(\langle 1 \rangle_0 = 1\) and
\(a_\emptyset = \prod_i c_i\) remains. 
\end{proof}

\renewcommand{\theremark}{4.3}
\begin{remark}[State of the Literature]\label{remark:4-3}
 For \(k = 2\) the
statement is known; it appears in \citep[Eq.~2.149]{DoranLasenby2003} in
the language of rotor composition. The generalization to arbitrary \(k\)
is an immediate consequence and may be considered folklore. We do not
present it as an own contribution but because it is needed for the
contrast with Theorem 5.2. In the prior work it appears as Theorem
\(S5.10\) \citep{Jenzer2026}.

\end{remark}

\subsection{The Star Condition Is Not
Dispensable}\label{the-star-condition-is-not-dispensable}

One might conjecture that mere pairwise anticommutation with unit
squares suffices for Theorem 4.2. This is false.

\renewcommand{\theremark}{4.4}
\begin{remark}[Counterexample]\label{remark:4-4}
 In \(\mathrm{Cl}_{0,3}\) with
generators \(e_1, e_2, e_3\), \(e_i^2 = -1\), consider the
\emph{triangle family}
\[B_1 = e_1 e_2, \qquad B_2 = e_1 e_3, \qquad B_3 = e_2 e_3.\] This is
an anticommuting unit family of type \(-1\): each \(B_i\) is a bivector
with \(B_i^2 = -1\), and any two of them share exactly one index, which
forces anticommutation. It possesses, however, no common generator, and
\[B_1 B_2 B_3 = (e_1 e_2)(e_1 e_3)(e_2 e_3) = -e_1 e_1 e_2 e_3 (e_2 e_3) = (e_2 e_3)(e_2 e_3) = -1,\]
so \(\langle B_{\{1,2,3\}} \rangle_0 = -1 \neq 0\). The star condition
is violated, and the chain formula breaks:
\[\langle P_3 \rangle_0 = c_1 c_2 c_3 + s_1 s_2 s_3 \langle B_1 B_2 B_3 \rangle_0 = c_1 c_2 c_3 - s_1 s_2 s_3 \;\neq\; \prod_{i=1}^3 \cos\frac{\theta_i}{2}\]
for all parameters with \(s_1 s_2 s_3 \neq 0\). The remaining part of
the computation stays valid: by Lemma 3.4 the terms with
\(|S| \in \{2\}\) vanish anyway, and \(|S| = 1\) yields bivectors
without scalar part.

It is exactly this counterexample that Theorem 5.2 survives. Chapter 5
shows why.

\begin{center}\rule{0.5\linewidth}{0.5pt}\end{center}

\pagebreak

\end{remark}

\section{The Closed Scalar Formula}\label{the-closed-scalar-formula}

\begin{quote}
This chapter contains the main theorem. The proof is a direct subset
expansion with two independent vanishing arguments. It requires no
induction and no star condition.
\end{quote}

\subsection{Sign Law}\label{sign-law}

\renewcommand{\thedefinition}{D5.1}
\begin{definition}[Period-Four Sign Law]\label{definition:D5-1}
 We set for
\(m \in \mathbb{N}_0\)
\[\eta(m) := (-1)^{m} \cdot (-1)^{m(m+1)/2} = (-1)^{m(m+3)/2}.\]

\end{definition}

\renewcommand{\thelemma}{5.1}
\begin{lemma}[Three Representations]\label{lemma:5-1}
 For all
\(m \in \mathbb{N}_0\),
\[\eta(m) = (-1)^{m(m+3)/2} = (-1)^{\lfloor m/2 \rfloor} = \operatorname{Re}\!\big[(1-\mathrm{i}) \, \mathrm{i}^m\big],\]
and \(\eta\) is periodic with period four:
\(\eta(0), \eta(1), \eta(2), \eta(3) = +1, +1, -1, -1\).

\end{lemma}

\begin{proof}
The first equality is the definition:
\(m + m(m+1)/2 = m(m+3)/2\). For the second it suffices to check \(m\)
modulo four, since \(m(m+3)/2\) modulo two depends only on \(m\) modulo
four: for \(m = 4t\), \(m(m+3)/2 = 2t(4t+3)\) is even, and
\(\lfloor m/2 \rfloor = 2t\) is even; for \(m = 4t+1\),
\(m(m+3)/2 = (4t+1)(2t+2)\) is even, and \(\lfloor m/2 \rfloor = 2t\) is
even; for \(m = 4t+2\), \(m(m+3)/2 = (2t+1)(4t+5)\) is odd, and
\(\lfloor m/2 \rfloor = 2t+1\) is odd; for \(m = 4t+3\),
\(m(m+3)/2 = (4t+3)(2t+3)\) is odd, and \(\lfloor m/2 \rfloor = 2t+1\)
is odd. In all four cases the parities agree. For the third equality one
computes \((1-\mathrm{i})\mathrm{i}^m\) for \(m = 0,1,2,3\):
\(1-\mathrm{i}\), \(1+\mathrm{i}\), \(-1+\mathrm{i}\),
\(-1-\mathrm{i}\), with real parts \(+1, +1, -1, -1\); since
\(\mathrm{i}^m\) has period four, the claim follows for all \(m\).

\end{proof}

\subsection{The Main Theorem}\label{the-main-theorem}

\renewcommand{\thetheorem}{5.2}
\begin{theorem}[Main Theorem, Loop Formula]\label{theorem:5-2}
 Let
\(B_1, \ldots, B_k\) be an arbitrary anticommuting unit family of type
\(-1\) in a real Clifford algebra, let
\(\theta_1, \ldots, \theta_k \in \mathbb{R}\), and let
\[\Theta_k = (U_1 \cdots U_k)(\widetilde U_1 \cdots \widetilde U_k), \qquad U_i = \cos\frac{\theta_i}{2} + B_i \sin\frac{\theta_i}{2}.\]
Then
\[\boxed{\;\langle \Theta_k \rangle_0 \;=\; \sum_{m=0}^{k} \eta(m) \sum_{\substack{S \subseteq \{1,\ldots,k\} \\ |S| = m}} \; \prod_{i \in S} \sin^2\frac{\theta_i}{2} \; \prod_{i \notin S} \cos^2\frac{\theta_i}{2}\;}\]
with \(\eta(m) = (-1)^{\lfloor m/2 \rfloor}\) by Lemma 5.1. No star
condition is assumed.

\end{theorem}

\begin{proof}
By Lemma 4.1,
\[\Theta_k = \Big(\sum_{S} a_S B_S\Big)\Big(\sum_{T} (-1)^{|T|} a_T B_T\Big) = \sum_{S} \sum_{T} (-1)^{|T|} a_S a_T \, B_S B_T,\]
and with the linearity of the scalar projection
\[\langle \Theta_k \rangle_0 = \sum_{S}\sum_{T} (-1)^{|T|} a_S a_T \, \langle B_S B_T \rangle_0. \tag{5.1}\]
We decompose the double sum into the diagonal part \(S = T\) and the
off-diagonal part \(S \neq T\).

\textbf{Diagonal part.} For \(S = T\) with \(|S| = m\), Lemma 3.2 gives
\[\langle B_S B_S \rangle_0 = \langle B_S^2 \rangle_0 = (-1)^{m(m+1)/2},\]
and the corresponding contribution is
\[(-1)^{m} a_S^2 (-1)^{m(m+1)/2} = \eta(m) \, a_S^2 = \eta(m) \prod_{i \in S} s_i^2 \prod_{i \notin S} c_i^2,\]
by Definition D5.1. Summed over all \(S\), ordered by cardinality, this
yields exactly the right-hand side of the claimed formula.

\textbf{Off-diagonal part.} Now let \(S \neq T\) and
\(D := S \triangle T \neq \emptyset\). By Lemma 3.3,
\(\langle B_S B_T \rangle_0 = \sigma(S,T) \langle B_D \rangle_0\). We
distinguish by the parity of \(|D|\).

\emph{Case 1: \(|D|\) even.} Then \(\langle B_D \rangle_0 = 0\) by Lemma
3.4, so the contribution of the pair \((S,T)\) vanishes individually.

\emph{Case 2: \(|D|\) odd.} Here the contribution does not vanish
individually in general, but it does pairwise. We group the ordered
pairs \((S,T)\) and \((T,S)\); since \(S \neq T\), these are two
distinct summands of (5.1). By Lemma 2.1,
\(\langle B_S B_T \rangle_0 = \langle B_T B_S \rangle_0\), so their sum
writes as
\[a_S a_T \langle B_S B_T \rangle_0 \Big[(-1)^{|T|} + (-1)^{|S|}\Big].\]
Now \(|S| + |T| = |S \triangle T| + 2|S \cap T| \equiv |D| \pmod 2\)
holds in general. Since \(|D|\) is odd, \(|S|\) and \(|T|\) have
different parity, hence \((-1)^{|S|} + (-1)^{|T|} = 0\) and the square
bracket vanishes.

In both cases the off-diagonal part contributes nothing, and the
diagonal part remains. \(\blacksquare\)

\emph{In words:} When expanding the closed word, \(2^k \times 2^k\)
terms arise. Only the \(2^k\) diagonal terms survive. Of the rest, some
vanish because a product of even length has no scalar part, and others
cancel in pairs. It is exactly this second group that is the reason why
Theorem 5.2 also holds for the triangle family from Remark 4.4: there
\(\langle B_{\{1,2,3\}} \rangle_0 = -1 \neq 0\), but \(|\{1,2,3\}|\) is
odd, so all pairs \((S,T)\) with \(S \triangle T = \{1,2,3\}\) vanish
pairwise.
\end{proof}

\renewcommand{\theremark}{5.3}
\begin{remark}\label{remark:5-3}
 The proof uses at no point that the \(B_i\) are
bivectors. It suffices that they are invertible, anticommute pairwise,
satisfy \(B_i^2 = -1\), and that their reversion is
\(\widetilde B_i = -B_i\). That bivectors achieve this is Definition
D2.2.

\end{remark}

\renewcommand{\theremark}{5.4}
\begin{remark}[Relation to the Prior Work]\label{remark:5-4}
 In the prior work,
Theorem 5.2 appears as Theorem \(S5.11\) \citep{Jenzer2026}, there
formulated for star families and with a different proof. The version
given here is stronger in two respects: it requires no star condition,
and it manages without induction over \(k\).

\end{remark}

\subsection{\texorpdfstring{5.3 First Check: the Case
\(k=1\)}{5.3 First Check: the Case k=1}}\label{first-check-the-case-k1}

For \(k = 1\), \(\Theta_1 = U_1 \widetilde U_1 = 1\), so
\(\langle \Theta_1 \rangle_0 = 1\). The formula gives
\[\eta(0) c_1^2 + \eta(1) s_1^2 = c_1^2 + s_1^2 = 1.\] Both agree. This
is the simplest, but not the least interesting check: it shows that the
sign law begins with \(\eta(1) = +1\) and not, as one might expect from
a naive reversion argument, with \(-1\).

\subsection{Arbitrary Index Orders}\label{arbitrary-index-orders}

\begin{quote}
This section generalizes the main theorem to an arbitrary permutation of
the index order in the second factor. The trivial reverse order from
Chapter 1.1 appears as the limiting case in which all signs are \(+1\).
\end{quote}

\renewcommand{\thedefinition}{D5.2}
\begin{definition}[Permuted Loop]\label{definition:D5-2}
 Let \(\tau\) be a permutation
of \(\{1, \ldots, k\}\). We set
\[Q_k^{(\tau)} = \widetilde U_{\tau(1)} \, \widetilde U_{\tau(2)} \cdots \widetilde U_{\tau(k)}, \qquad \Theta_k(\tau) = P_k \, Q_k^{(\tau)},\]
so that \(\Theta_k(\mathrm{id}) = \Theta_k\) and
\(\Theta_k(\tau_{\mathrm{rev}})\) for the reverse permutation
\(\tau_{\mathrm{rev}}(j) = k+1-j\) gives the trivial product
\(P_k \widetilde{P_k} = 1\) from Chapter 1.1. For
\(S \subseteq \{1, \ldots, k\}\) let
\[\iota(\tau, S) := \#\{(i, j) \in S \times S : i < j,\ \tau^{-1}(i) > \tau^{-1}(j)\}\]
be the \emph{inversion count} of \(\tau\) on \(S\).

\end{definition}

\renewcommand{\thetheorem}{5.5}
\begin{theorem}[Permuted Loop Formula]\label{theorem:5-5}
 Let \(B_1, \ldots, B_k\)
be an anticommuting unit family of type \(-1\) that satisfies the star
condition, and let \(\tau\) be a permutation of \(\{1, \ldots, k\}\).
Then
\[\boxed{\;\langle \Theta_k(\tau) \rangle_0 = \sum_{S \subseteq \{1,\ldots,k\}} \eta(|S|) \, (-1)^{\iota(\tau, S)} \prod_{i \in S} w_i \prod_{i \notin S} (1 - w_i).\;}\]

\end{theorem}

\begin{proof}
The same subset expansion as in Lemma 4.1, applied to the
product ordered by \(\tau\), gives
\[Q_k^{(\tau)} = \sum_{T} (-1)^{|T|} a_T \, B_{T, \tau},\] where
\(B_{T, \tau}\) denotes the product of the \(B_i\) with \(i \in T\) in
the order given by \(\tau\). Reordering into ascending order costs
exactly the inversion count:
\(B_{T, \tau} = (-1)^{\iota(\tau, T)} B_T\). Thus
\[\langle \Theta_k(\tau) \rangle_0 = \sum_{S} \sum_{T} (-1)^{|T| + \iota(\tau, T)} \, a_S a_T \, \langle B_S B_T \rangle_0.\]
For \(S \neq T\) and \(D = S \triangle T \neq \emptyset\), every term
vanishes individually: for even \(|D|\) by Lemma 3.4, for odd \(|D|\) by
the star condition. (Thus the weaker requirement
\(\langle B_D \rangle_0 = 0\) for all \(D\) of odd cardinality
suffices.) The diagonal part \(S = T\) with \(|S| = m\) remains, and its
contribution is
\[(-1)^{m + \iota(\tau, S)} \, a_S^2 \, \langle B_S^2 \rangle_0 = \eta(m) \, (-1)^{\iota(\tau, S)} \, a_S^2\]
by Lemma 3.2 and Definition D5.1. \(\blacksquare\)

\emph{In words:} The permutation acts not globally, but subset by
subset: each inversion within a subset flips its sign. What is decisive
is thus the inversion parity on each subset, not the cycle structure of
the permutation.
\end{proof}

\renewcommand{\thecorollary}{5.6}
\begin{corollary}[The Reverse Order as a Limiting Case]\label{corollary:5-6}
 For the
reverse permutation \(\tau_{\mathrm{rev}}\),
\(\iota(\tau_{\mathrm{rev}}, S) = \binom{|S|}{2}\), and since
\[\eta(m) \, (-1)^{m(m-1)/2} = (-1)^{m(m+3)/2 \, + \, m(m-1)/2} = (-1)^{m^2 + m} = 1 \qquad \text{for all } m\]
Theorem 5.5 gives
\[\langle \Theta_k(\tau_{\mathrm{rev}}) \rangle_0 = \sum_{S} \prod_{i \in S} w_i \prod_{i \notin S} (1 - w_i) = \prod_{i=1}^{k} \big(w_i + (1 - w_i)\big) = 1.\]
The permuted loop formula thus reproduces the scalar part of the trivial
identity \(P_k \widetilde{P_k} = 1\) as the limiting case in which all
signs are \(+1\). The element-wise identity is of course stronger and
remains the one-line computation from Chapter 1.1.

\end{corollary}

\renewcommand{\theremark}{5.7}
\begin{remark}[On the Scope, with Numerical Evidence]\label{remark:5-7}
 The
hypotheses of Theorem 5.5 can be sharpened at two points. First, for
\(\tau = \mathrm{id}\) no star condition is needed: that is Theorem 5.2,
whose pairing argument cancels the odd off-diagonal terms in pairs.
Second, this pairing argument fails termwise for
\(\tau \neq \mathrm{id}\). Nevertheless, numerical evidence suggests
that the formula also holds without the star condition in that case:
tested on the triangle family from Remark 4.4, which violates the star
condition, under all six permutations of \(\{1, 2, 3\}\) with random
parameter sets; the largest deviation encountered was in the range of
double machine precision. A general proof would have to capture the
cancellation of the odd off-diagonal terms in coarser groups than pairs;
we have not determined such a grouping.

\begin{center}\rule{0.5\linewidth}{0.5pt}\end{center}

\pagebreak

\end{remark}

\section{Equivalent Forms, Bounds and
Checks}\label{equivalent-forms-bounds-and-checks}

\begin{quote}
This chapter reshapes the main theorem. The compact form brings it onto
a single line, the polar form makes the zeros visible, and the check
cases allow a verification by hand.
\end{quote}

\subsection{The Subset Formula in
Weights}\label{the-subset-formula-in-weights}

With \(w_i = \sin^2\frac{\theta_i}{2}\), hence
\(\cos^2\frac{\theta_i}{2} = 1 - w_i\), Theorem 5.2 reads
\[\langle \Theta_k \rangle_0 = \sum_{S \subseteq \{1,\ldots,k\}} \eta(|S|) \prod_{i \in S} w_i \prod_{i \notin S} (1 - w_i). \tag{6.1}\]
Without the sign law, the right-hand side would be a sum over a
probability distribution and identically \(1\). The sign law is thus
exactly the deviation that separates \(\langle \Theta_k \rangle_0\) from
the trivial quantity \(1\). We call (6.1) the \emph{subset formula}.

\subsection{The Compact Form}\label{the-compact-form}

\renewcommand{\thecorollary}{6.1}
\begin{corollary}[Compact Form]\label{corollary:6-1}
 For all \(k \geq 1\) and all
\(\theta_i \in \mathbb{R}\),
\[\boxed{\;\langle \Theta_k \rangle_0 = \operatorname{Re}\!\left[(1 - \mathrm{i}) \prod_{i=1}^{k} \left(\cos^2\frac{\theta_i}{2} + \mathrm{i} \, \sin^2\frac{\theta_i}{2}\right)\right] = \operatorname{Re}\!\left[(1-\mathrm{i}) \prod_{i=1}^{k} \big((1 - w_i) + \mathrm{i} \, w_i\big)\right].\;}\]

\end{corollary}

\begin{proof}
By Lemma 5.1,
\(\eta(m) = \operatorname{Re}[(1-\mathrm{i})\mathrm{i}^m]\). Since
\((1-\mathrm{i})\) and the weights \(w_i\) are independent of \(m\) and
\(\operatorname{Re}\) is real-linear, it follows from (6.1) that
\[\langle \Theta_k \rangle_0 = \operatorname{Re}\!\left[(1-\mathrm{i}) \sum_{S} \mathrm{i}^{|S|} \prod_{i \in S} w_i \prod_{i \notin S} (1-w_i)\right] = \operatorname{Re}\!\left[(1-\mathrm{i}) \prod_{i=1}^{k} \big((1-w_i) + \mathrm{i} w_i\big)\right],\]
where the last step is the factorization of the subset sum, i.e., the
reverse of expanding. \(\blacksquare\)

\emph{In short:} The entire statement resides in a product of complex
numbers and a real-part projection at the end. The individual factors
each depend on only one parameter; the real-part projection prevents the
result from decomposing into a product of real factors.
\end{proof}

\renewcommand{\thecorollary}{6.2}
\begin{corollary}[Non-Separability]\label{corollary:6-2}
 For \(k \geq 2\) there exists
no function \(g : [0,1] \to \mathbb{R}\) with
\[\langle \Theta_k \rangle_0 = \prod_{i=1}^{k} g(w_i) \qquad \text{for all } w_1, \ldots, w_k \in [0,1].\]

\end{corollary}

\begin{proof}
Setting \(w_3 = \cdots = w_k = 0\) in (6.1), only subsets
\(S \subseteq \{1,2\}\) contribute, and with \(\eta(0) = \eta(1) = +1\),
\(\eta(2) = -1\) one obtains
\[1 \cdot (1-w_1)(1-w_2) + 1 \cdot \big[w_1(1-w_2) + (1-w_1)w_2\big] - 1 \cdot w_1 w_2 = 1 - 2 w_1 w_2.\]
If the quantity were separable, then
\(1 - 2 w_1 w_2 = g(w_1) g(w_2) g(0)^{k-2}\) would hold for all
\(w_1, w_2\). Setting \(w_2 = 0\) gives \(1 = g(w_1) g(0)^{k-1}\) for
all \(w_1\), so \(g\) is constant on \([0,1]\) and so is the right-hand
side. But the left-hand side equals \(-1\) for \(w_1 = w_2 = 1\) and
\(+1\) for \(w_1 = w_2 = 0\), hence is not constant. Contradiction.
\(\blacksquare\)

Corollary 6.2 is the precise formulation of the observation that
motivated the prior work: the parameters of the individual family
members couple, and inevitably so as soon as more than one is present.
For the chain this does not hold, by Theorem 4.2.
\end{proof}

\subsection{The Polar Form}\label{the-polar-form}

\renewcommand{\thecorollary}{6.3}
\begin{corollary}[Polar Form]\label{corollary:6-3}
 Set for \(i = 1, \ldots, k\)
\[\rho_i := \sqrt{\cos^4\frac{\theta_i}{2} + \sin^4\frac{\theta_i}{2}}\,, \qquad \alpha_i := \arctan\!\left(\tan^2\frac{\theta_i}{2}\right) \in \left[0, \tfrac{\pi}{2}\right).\]
Then
\[\boxed{\;\langle \Theta_k \rangle_0 = \sqrt{2} \left(\prod_{i=1}^{k} \rho_i\right) \cos\!\left(\sum_{i=1}^{k} \alpha_i - \frac{\pi}{4}\right).\;}\]
Furthermore, the identity
\[\rho_i^2 = 1 - 2 w_i (1 - w_i) = 1 - \tfrac{1}{2}\sin^2\theta_i, \qquad \text{in particular} \quad \tfrac{1}{\sqrt 2} \le \rho_i \le 1.\]

\end{corollary}

\begin{proof}
The \(i\)-th factor of the compact form is
\(z_i = (1-w_i) + \mathrm{i} w_i = \cos^2\frac{\theta_i}{2} + \mathrm{i}\sin^2\frac{\theta_i}{2}\).
Both real and imaginary parts are non-negative, so \(z_i\) lies in the
first quadrant, and \(z_i \neq 0\) since real and imaginary parts do not
vanish simultaneously. Hence \(z_i = \rho_i e^{\mathrm{i}\alpha_i}\)
with
\(\rho_i = |z_i| = \sqrt{\cos^4\frac{\theta_i}{2} + \sin^4\frac{\theta_i}{2}}\)
and \(\alpha_i = \arg z_i \in [0, \frac{\pi}{2})\), where
\(\tan \alpha_i = \sin^2\frac{\theta_i}{2} / \cos^2\frac{\theta_i}{2} = \tan^2\frac{\theta_i}{2}\)
for \(\cos\frac{\theta_i}{2} \neq 0\) and \(\alpha_i \to \frac{\pi}{2}\)
in the limit. Further,
\(1 - \mathrm{i} = \sqrt 2\, e^{-\mathrm{i}\pi/4}\). Thus
\[(1-\mathrm{i})\prod_i z_i = \sqrt 2 \left(\prod_i \rho_i\right) \exp\!\left[\mathrm{i}\left(\sum_i \alpha_i - \frac{\pi}{4}\right)\right],\]
and the real part is the claimed cosine form. For the identity: with
\(c^2 + s^2 = 1\),
\[\rho^2 = c^4 + s^4 = (c^2+s^2)^2 - 2 c^2 s^2 = 1 - 2 w (1-w),\] and
since \(2cs = \sin\theta\),
\(2c^2s^2 = \frac{1}{2}(2cs)^2 = \frac{1}{2}\sin^2\theta\). Since
\(w(1-w) \in [0, \frac14]\), \(\rho^2 \in [\frac12, 1]\) follows.
\(\blacksquare\)

\emph{In words:} The order defect is the product of the individual
magnitudes, multiplied by a single cosine in whose argument the angles
\(\alpha_i\) are \emph{added}. The nonlinearity resides entirely in this
addition under a cosine. The term \(-\frac{\pi}{4}\) is the only global
contribution; it stems from the prefactor \(1-\mathrm{i}\) and is the
actual reason for the period four in the sign law.
\end{proof}

\renewcommand{\thecorollary}{6.4}
\begin{corollary}[Bound]\label{corollary:6-4}
 For all \(k\) and all parameters,
\[\big|\langle \Theta_k \rangle_0\big| \;\le\; \min\!\left(1, \; \sqrt 2 \prod_{i=1}^{k} \rho_i\right), \qquad \text{and} \qquad 2^{(1-k)/2} \le \sqrt 2 \prod_{i=1}^{k}\rho_i \le \sqrt 2 .\]

\end{corollary}

\begin{proof}
The second bound follows from Corollary 6.3, since
\(|\cos| \le 1\), together with \(\rho_i \in [\frac{1}{\sqrt2}, 1]\).
For the first, the bound \(1\): for every even basis blade \(e_A\) in
\(\mathrm{Cl}_{0,n}\), i.e., \(|A|\) even,
\[\langle e_A \widetilde{e_A} \rangle_0 = (-1)^{|A|(|A|-1)/2} \, e_A^2 = (-1)^{|A|(|A|-1)/2} \, (-1)^{|A|(|A|+1)/2} = (-1)^{|A|^2} = 1,\]
and for \(A \neq B\), \(\langle e_A \widetilde{e_B} \rangle_0 = 0\),
since \(e_A \widetilde{e_B}\) is proportional to \(e_{A \triangle B}\).
On the even subalgebra \(\mathrm{Cl}_{0,n}^{+}\), therefore,
\[x \mapsto \langle x \widetilde x \rangle_0 = \sum_{A \text{ even}} x_A^2\]
is a positive definite quadratic form. On the full algebra this does not
hold --- for \(x = e_1\),
\(\langle x \widetilde x \rangle_0 = e_1^2 = -1\) ---; the signature
dependence of this form is found in \citep{Lounesto2001}. Since \(\Theta_k\),
as a product of rotors, is itself a rotor, \(\Theta_k\) lies in
\(\mathrm{Cl}_{0,n}^{+}\), and
\(\langle \Theta_k \widetilde{\Theta_k} \rangle_0 = 1\), hence
\(\sum_{A \text{ even}} (\Theta_k)_A^2 = 1\) and in particular
\(|\langle \Theta_k \rangle_0| \le 1\). 
\end{proof}

\subsection{Complete Zero Criterion}\label{complete-zero-criterion}

\renewcommand{\thecorollary}{6.5}
\begin{corollary}[Zeros]\label{corollary:6-5}
 We have
\[\langle \Theta_k \rangle_0 = 0 \quad \Longleftrightarrow \quad \sum_{i=1}^{k} \alpha_i \in \frac{3\pi}{4} + \pi \mathbb{Z}.\]
For \(k = 1\) there is no zero; for every \(k \geq 2\) zeros exist.

\end{corollary}

\begin{proof}
By Corollary 6.3 and \(\rho_i > 0\),
\(\langle \Theta_k \rangle_0\) vanishes if and only if
\(\cos(\sum \alpha_i - \frac{\pi}{4}) = 0\), i.e.,
\(\sum \alpha_i - \frac{\pi}{4} \in \frac{\pi}{2} + \pi\mathbb{Z}\),
that is \(\sum \alpha_i \in \frac{3\pi}{4} + \pi\mathbb{Z}\).

Since \(\alpha_i \in [0, \frac{\pi}{2})\), \(\sum_i \alpha_i\) lies in
the interval \([0, \frac{k\pi}{2})\). For \(k=1\),
\(\frac{\pi}{2} < \frac{3\pi}{4}\), so no value of the condition is
reached; this agrees with \(\langle \Theta_1 \rangle_0 = 1\) from
Chapter 5.3. For \(k \geq 2\),
\(\frac{k\pi}{2} \geq \pi > \frac{3\pi}{4}\); since \(\alpha_i\) depends
continuously on \(\theta_i\) and traverses the entire range
\([0, \frac{\pi}{2})\), \(\sum_i \alpha_i\) assumes every value in
\([0, \frac{k\pi}{2})\), in particular \(\frac{3\pi}{4}\).
\(\blacksquare\)

\pagebreak
\end{proof}

\subsection{Check Cases}\label{check-cases}

We write throughout \(w_i = \sin^2\frac{\theta_i}{2}\).

\begin{longtable}{cl}
\toprule
$k$ & $\langle \Theta_k \rangle_0$ \\
\midrule
\endhead
\bottomrule
$1$ & $1$ \\
$2$ & $1 - 2 w_1 w_2$ \\
$3$ & $1 - 2(w_1 w_2 + w_1 w_3 + w_2 w_3) + 4 w_1 w_2 w_3$ \\
$4$ & $1 - 2\!\!\sum\limits_{i<j}\!\! w_i w_j + 4\!\!\sum\limits_{i<j<l}\!\! w_i w_j w_l - 4 w_1 w_2 w_3 w_4$ \\
\end{longtable}

These expressions are obtained from (6.1) by expanding and are
symbolically confirmed in Appendix A. Two checks by hand:

\begin{itemize}
\tightlist
\item
  For \(\theta_i = 0\) for all \(i\), i.e., \(w_i = 0\),
  \(\Theta_k = 1\) and the formula gives \(1\).
\item
  For \(\theta_i = \pi\) for all \(i\), i.e., \(w_i = 1\),
  \(U_i = B_i\), \(\widetilde U_i = -B_i\), so
  \(\Theta_k = (-1)^k B_{\{1..k\}}^2 = (-1)^k (-1)^{k(k+1)/2} = \eta(k)\);
  and in (6.1) exactly the term \(S = \{1,\ldots,k\}\) with coefficient
  \(\eta(k)\) survives. For \(k = 3\) both give \(\eta(3) = -1\), in
  agreement with the table: \(1 - 2\cdot 3 + 4 = -1\).
\end{itemize}

\subsection{The Symmetric Case}\label{the-symmetric-case}

\renewcommand{\thecorollary}{6.6}
\begin{corollary}[Symmetric Case]\label{corollary:6-6}
 If all parameters are equal,
\(\theta_i = \theta\) and \(w := \sin^2\frac{\theta}{2}\), then
\[\langle \Theta_k \rangle_0 = \sum_{m=0}^{k} (-1)^{\lfloor m/2 \rfloor} \binom{k}{m} w^m (1-w)^{k-m} = \operatorname{Re}\!\Big[(1-\mathrm{i})\big(1 + (\mathrm{i}-1)w\big)^{k}\Big] = \sqrt 2 \, \rho^{k} \cos\big(k\alpha - \tfrac{\pi}{4}\big).\]
In particular, for \(\theta = \frac{\pi}{2}\), i.e., \(w = \frac12\),
\(\rho = \frac{1}{\sqrt2}\), \(\alpha = \frac{\pi}{4}\):
\[\langle \Theta_k \rangle_0 = 2^{(1-k)/2} \cos\!\left(\frac{(k-1)\pi}{4}\right),\]
which vanishes for \(k \equiv 3 \pmod 4\), first at \(k = 3\).

\end{corollary}

\begin{proof}
The first equality is (6.1) with equal weights, where
\(\binom{k}{m}\) subsets of cardinality \(m\) exist. The second is
Corollary 6.1 with \(z_i = 1 + (\mathrm{i}-1)w\) for all \(i\), the
third Corollary 6.3. For \(w = \frac12\),
\(z = \frac12 + \frac{\mathrm{i}}{2}\), so \(\rho = \frac{1}{\sqrt2}\)
and \(\alpha = \frac{\pi}{4}\); substituting gives
\(\sqrt2 \cdot 2^{-k/2}\cos(\frac{k\pi}{4} - \frac{\pi}{4})\).

\end{proof}

\renewcommand{\thecorollary}{6.7}
\begin{corollary}[Lacunary Coefficient Rule]\label{corollary:6-7}
 In the symmetric
case, \(\langle \Theta_k \rangle_0\) is a polynomial in \(w\) of degree
at most \(k\), and the coefficient of \(w^j\) is
\[[w^j]\,\langle \Theta_k \rangle_0 = \binom{k}{j} (-1)^{j} \, 2^{(j+1)/2} \cos\!\left(\frac{\pi(j+1)}{4}\right).\]
It vanishes if and only if \(j \equiv 1 \pmod 4\).

\end{corollary}

\begin{proof}
From Corollary 6.6 and the binomial theorem,
\[\langle \Theta_k \rangle_0 = \operatorname{Re}\!\left[(1-\mathrm{i}) \sum_{j=0}^{k} \binom{k}{j} (\mathrm{i}-1)^j w^j\right] = \sum_{j=0}^{k} \binom{k}{j} w^j \operatorname{Re}\!\left[(1-\mathrm{i})(\mathrm{i}-1)^j\right].\]
Now \((\mathrm{i}-1)^j = (-1)^j (1-\mathrm{i})^j\), so
\((1-\mathrm{i})(\mathrm{i}-1)^j = (-1)^j (1-\mathrm{i})^{j+1}\). With
\(1-\mathrm{i} = \sqrt2 e^{-\mathrm{i}\pi/4}\),
\((1-\mathrm{i})^{j+1} = 2^{(j+1)/2} e^{-\mathrm{i}\pi(j+1)/4}\)
follows, and from this the stated form of the coefficient. It vanishes
if and only if \(\cos\frac{\pi(j+1)}{4} = 0\), i.e.,
\(\frac{j+1}{4} \in \frac12 + \mathbb{Z}\), that is
\(j + 1 \equiv 2 \pmod 4\), hence \(j \equiv 1 \pmod 4\).
\(\blacksquare\)

\emph{In short:} In the symmetric case, every fourth term of the
polynomial is missing, beginning with the linear one. In particular,
\(\langle \Theta_k \rangle_0 = 1 + O(w^2)\) for every \(k\): the order
defect sets in at second order, not first. Such lacunary binomial sums
are classical; the associated technique of real-part projection from
roots of unity goes back to Ramus and is presented in modern form in
\citep{Dobson2017}.
\end{proof}

\subsection{The Hyperbolic Case}\label{the-hyperbolic-case}

The method carries over to families with square \(+1\). We note the
result because it shows how far the proof carries, and because it
explains the role of the number \(1-\mathrm{i}\).

\renewcommand{\thetheorem}{6.8}
\begin{theorem}[Hyperbolic Case]\label{theorem:6-8}
 Let \(B_1, \ldots, B_k\) be
bivectors of a real Clifford algebra with
\[B_i^2 = +1, \qquad B_i B_j = -B_j B_i \ (i \neq j),\] e.g.,
\(B_i = C S_i\) with \(C^2 = -1\), \(S_i^2 = +1\) and pairwise
anticommutation. Set for \(\chi_i \in \mathbb{R}\)
\[U_i = \cosh\frac{\chi_i}{2} + B_i \sinh\frac{\chi_i}{2}, \qquad \widetilde U_i = \cosh\frac{\chi_i}{2} - B_i \sinh\frac{\chi_i}{2},\]
and
\(\Theta_k = (U_1 \cdots U_k)(\widetilde U_1 \cdots \widetilde U_k)\).
Then
\[\langle \Theta_k \rangle_0 = \sum_{m=0}^{k} \eta^{+}(m) \sum_{|S| = m} \prod_{i \in S} \sinh^2\frac{\chi_i}{2} \prod_{i \notin S} \cosh^2\frac{\chi_i}{2} = \operatorname{Re}\!\left[(1+\mathrm{i}) \prod_{i=1}^{k}\left(\cosh^2\frac{\chi_i}{2} + \mathrm{i}\sinh^2\frac{\chi_i}{2}\right)\right]\]
with
\[\eta^{+}(m) = (-1)^{m(m+1)/2} = \operatorname{Re}\!\big[(1+\mathrm{i})\,\mathrm{i}^m\big],\]
i.e., the sign sequence \(+, -, -, +, +, -, -, \ldots\)

\end{theorem}

\begin{proof}
The computation is verbatim that of Theorem 5.2, with
three adjustments. First, the recursion of Lemma 3.2 here gives
\(B_S^2 = (-1)^{m-1}(+1) B_{S'}^2\), so
\(B_S^2 = (-1)^{0+1+\cdots+(m-1)} = (-1)^{m(m-1)/2}\), whence
\(\eta^{+}(m) = (-1)^m (-1)^{m(m-1)/2} = (-1)^{m(m+1)/2}\). Second,
Lemma 3.4 remains valid: \(B_i\) is invertible with \(B_i^{-1} = B_i\),
and the conjugation argument is unchanged. Third, Lemma 3.3 remains
valid, since only anticommutation and invertibility enter. The subset
expansion of \(U_i = \cosh\frac{\chi_i}{2} + B_i \sinh\frac{\chi_i}{2}\)
is structurally identical to Lemma 4.1. For the compact form one checks
\(\operatorname{Re}[(1+\mathrm{i})\mathrm{i}^m]\) for \(m = 0,1,2,3\):
\(1+\mathrm{i}\), \(-1+\mathrm{i}\), \(-1-\mathrm{i}\),
\(1-\mathrm{i}\), with real parts \(+1, -1, -1, +1\); this agrees with
\((-1)^{m(m+1)/2}\), and the factorization proceeds as in Corollary 6.1.

\end{proof}

\renewcommand{\theremark}{6.9}
\begin{remark}\label{remark:6-9}
 The only difference between the two cases is the
global prefactor, \(1-\mathrm{i}\) versus \(1+\mathrm{i}\), i.e., a
rotation by \(\pm\frac{\pi}{4}\). The sign sequence of one case is the
other shifted by one position. Since in the hyperbolic case
\(\alpha_i = \arctan(\tanh^2\frac{\chi_i}{2}) \in [0, \frac{\pi}{4})\)
and the zero condition reads
\(\sum_i \alpha_i \in \frac{\pi}{4} + \pi\mathbb{Z}\), here too \(k=1\)
is zero-free and every \(k \geq 2\) is not.

\end{remark}

\subsection{The Mixed Case}\label{the-mixed-case}

\begin{quote}
Theorem 5.2 treats families of type \(-1\), Theorem 6.8 those of type
\(+1\). Both are special cases of a single theorem in which each family
member carries its own type.
\end{quote}

\renewcommand{\thetheorem}{6.10}
\begin{theorem}[Mixed Families]\label{theorem:6-10}
 Let \(B_1, \ldots, B_k\) be
bivectors of a real Clifford algebra with
\[B_i^2 = \varepsilon_i \in \{-1, +1\}, \qquad B_i B_j = -B_j B_i \quad (i \neq j),\]
e.g., \(B_i = C S_i\) with \(C^2 = -1\), \(S_i^2 = \varepsilon_i\) and
pairwise anticommutation. For each \(i\) let a parameter be given,
\(\theta_i\) for \(\varepsilon_i = -1\) or \(\chi_i\) for
\(\varepsilon_i = +1\), and
\[U_i = \begin{cases} \cos\dfrac{\theta_i}{2} + B_i \sin\dfrac{\theta_i}{2}, & \varepsilon_i = -1,\\[6pt] \cosh\dfrac{\chi_i}{2} + B_i \sinh\dfrac{\chi_i}{2}, & \varepsilon_i = +1, \end{cases}\]
as well as
\(\Theta_k = (U_1 \cdots U_k)(\widetilde U_1 \cdots \widetilde U_k)\).
Setting \(w_i = \sin^2\frac{\theta_i}{2}\) or
\(w_i = \sinh^2\frac{\chi_i}{2}\),
\[\boxed{\;\langle \Theta_k \rangle_0 = \sum_{S \subseteq \{1,\ldots,k\}} (-1)^{|S|(|S|+1)/2} \prod_{i \in S} \varepsilon_i \, w_i \prod_{i \notin S} (1 + \varepsilon_i w_i).\;}\]

\end{theorem}

\begin{proof}
The computation is verbatim that of Theorem 5.2 or Theorem
6.8, with a single addition: each family member carries its own sign.
The recursion of Lemma 3.2 here gives
\[B_S^2 = (-1)^{|S|(|S|-1)/2} \prod_{i \in S} \varepsilon_i;\] the
parity lemma 3.4 remains valid, since \(B_i\) is invertible with
\(B_i^{-1} = \varepsilon_i B_i\) and the conjugation argument runs
unchanged; Lemma 3.3 and the pairing argument remain unchanged as well.
The diagonal contribution of \(S\) with \(|S| = m\) is thus
\[(-1)^{m} \, a_S^2 \, \langle B_S^2 \rangle_0 = (-1)^{m(m+1)/2} \prod_{i \in S} \varepsilon_i \; a_S^2,\]
where
\(a_S^2 = \prod_{i \in S} w_i \prod_{i \notin S} (1 + \varepsilon_i w_i)\),
since for \(\varepsilon_i = -1\) we have \(\cos^2 = 1 - \sin^2\) and for
\(\varepsilon_i = +1\) we have \(\cosh^2 = 1 + \sinh^2\). Thus the sum
of the diagonal contributions yields exactly the claimed subset formula.

\end{proof}

\renewcommand{\theremark}{6.11}
\begin{remark}[The Pure Cases as Special Cases]\label{remark:6-11}
 For
\(\varepsilon_i = +1\) for all \(i\), the subset formula of Theorem 6.10
is verbatim that of Theorem 6.8. For \(\varepsilon_i = -1\) for all
\(i\), Theorem 5.2 arises: the factor
\[(-1)^{|S|(|S|+1)/2} \prod_{i \in S} (-1) = (-1)^{|S|(|S|+3)/2} = \eta(|S|)\]
is then exactly the sign law from Definition D5.1, and outside \(S\),
\(1+\varepsilon_i w_i = 1-w_i\). The two pure cases each possess their
own particularly simple compact form, Corollary 6.1 or Theorem 6.8. For
mixed families, the subset formula provides a uniform closed
representation.

\end{remark}

\renewcommand{\theremark}{6.12}
\begin{remark}[No Bound]\label{remark:6-12}
 A bound as in Corollary 6.4 does not
exist in the mixed case: the hyperbolic weights
\(w_i = \sinh^2\frac{\chi_i}{2}\) are unbounded, and
\(|\langle \Theta_k \rangle_0|\) grows with them beyond any value. The
polar form and the zero criterion of Chapters 6.3 and 6.4 therefore do
not carry over directly.

\begin{center}\rule{0.5\linewidth}{0.5pt}\end{center}

\pagebreak

\end{remark}

\section{Why All This}\label{why-all-this}

\begin{quote}
This chapter justifies the work, names its computational utility, its
possible points of connection, and openly its origin. It contains no
theorems.
\end{quote}

\subsection{What Is Measured}\label{what-is-measured}

The quantity \(\langle \Theta_k \rangle_0\) is a scalar measure for a
very specific effect: that order matters in a rotor product. It is
constructed so that it would be identically \(1\) in the commutative
situation and is trivially \(1\) in the reverse-ordered situation. What
it captures is exclusively the difference between ``first \(U_1\), then
\(U_2\), then back in the same order'' and ``first \(U_1\), then
\(U_2\), then back in reverse order.''

The family of pairwise anticommuting unit bivectors is the extreme case
for this. Anticommutation is the maximal possible non-commutativity
between two elements with unit square: not a part, but the entire
product reverses sign. One thus obtains the order defect not in an
arbitrary, but in the sharpest configuration, and obtains it there in
closed form for arbitrarily many factors. This is the mathematical
content of the work.

Three features of the result appear remarkable to us, independent of any
application.

First, the \emph{structure of the coupling}. The compact form shows that
each family member contributes exactly one complex factor and that all
coupling stems solely from the real-part projection at the end.
Corollary 6.2 turns this into a statement: separability is impossible
for \(k \geq 2\), and not because of an unfavorable parametrization but
structurally. The polar form additionally says where the coupling
resides: in a sum of angles under a single cosine.

Second, the \emph{contrast with the chain}. For the same family, the
scalar part of the simple chain factorizes completely,
\(\langle P_k \rangle_0 = \prod_i \cos\frac{\theta_i}{2}\), under the
star condition. There are thus two natural scalar quantities for the
same configuration, and they behave oppositely: one separates, the other
does not. Whoever infers independence of the parameters from the one is
refuted by the other. Independence is thus not a property of the family
but of the observed quantity. We consider this the actual lesson of the
computation.

Third, the \emph{unexpected distribution of robustness}. The simpler
formula requires an additional hypothesis, the more complicated one does
not. Remark 4.4 gives a family for which Theorem 4.2 is false and
Theorem 5.2 is true. The reason is readable in the proof of Theorem 5.2:
the perturbation terms of the loop come in pairs whose signs cancel for
a parity reason, while the chain has only individual terms and therefore
carries every perturbation term unprotected.

\subsection{Computational Utility}\label{computational-utility}

The most immediate utility of the work is a complexity argument, and it
can be quantified exactly.

Whoever determines \(\langle \Theta_k \rangle_0\) by direct expansion
must, by (5.1), evaluate a double sum over all subset pairs, i.e.,
\[2^k \cdot 2^k = 4^k\] products of subset products, each with a sign
determination. For \(k = 10\) that is \(1{,}048{,}576\) terms, for
\(k = 20\) already \(1{,}099{,}511{,}627{,}776\), i.e., about
\(10^{12}\). The compact form from Corollary 6.1 needs, by contrast,
\[k \text{ complex multiplications},\] i.e., for \(k = 20\) exactly
twenty, and one extraction of the real part. The effort drops from
exponential to linear. The polar form from Corollary 6.3 further reduces
to \(k\) square roots, \(k\) arctangent evaluations, one sum, and one
cosine, which is often the better form for error estimates and for
statements about sign changes.

This is not a mere bookkeeping advantage. It makes the quantity
accessible to symbolic and numerical systems of geometric algebra in the
first place, especially in regimes with many factors, where direct
evaluation is not only slow but also numerically unreliable due to sign
catastrophes. Corollary 6.7 additionally provides a structural control
for implementations: in the symmetric case, every coefficient with
\(j \equiv 1 \pmod 4\) must vanish exactly. Whoever wants to test an
implementation thereby has a sharp test that needs neither reference
values nor tolerances.

\subsection{Where Such Families
Occur}\label{where-such-families-occur}

The configuration is not artificial. We name four points of connection,
in descending certainty, and explicitly mark where we make no claim.

\textbf{Geometrically: common direction in the rotation planes.} The
condition \(B_i = C S_i\) means that all rotation planes contain the
vector \(C\). In \(\mathrm{Cl}_{3,0}\) with \(C = e_1\), \(S_1 = e_2\),
\(S_2 = e_3\) one obtains \(B_1 = e_1 e_2\) and \(B_2 = e_1 e_3\); these
are rotations about the mutually perpendicular axes \(e_3\) and \(e_2\).
These two anticommute and have square \(-1\), so Theorem 5.2 is
applicable. Worth noting is the contrast with the complete axis set
\(\{e_2 e_3,\, e_1 e_3,\, e_1 e_2\}\) of a Cardan suspension: that is
not a star family but the triangle family from Remark 4.4, for which
Theorem 4.2 becomes false while Theorem 5.2 holds.

\textbf{Algebraically: the existence question for large \(k\).} How many
pairwise anticommuting unit bivectors a given algebra carries is a
classical question. Via the star construction, the answer for
\(\mathrm{Cl}_{0,k+1}\) is trivially positive for every \(k\), so the
theorems of this work know no dimensional bound. If one instead wants to
realize anticommuting families in a \emph{given}, small algebra or in a
matrix algebra of small size, one enters the domain of Hurwitz-Radon
theory and its relatives \citep{Eckmann1994,Hrubes2016}. We make no own
statement on this.

\textbf{Operator-theoretically: anticommuting sets.} Sets of pairwise
anticommuting operators with unit square occur as a basic building block
in several fields, for instance as Majorana or generalized Pauli
systems, and their combinatorial structure is the subject of dedicated
investigations \citep{Curtright2014,Shirokov2023,SarkarVanDenBerg2019}.
Where ordered products and their traces
occur there, the scalar part in the sense of this work is precisely the
normalized trace, as Appendix A exploits in its matrix check. We
consider it plausible that the closed form is useful there, but claim no
concrete application.

\textbf{As potential points of connection} we name, explicitly without
own results: rotor chains in kinematics and robotics, where ordered
products of rotations about intersecting axes occur; methods of
geometric algebra in signal processing and computer vision, where scalar
projections of rotor products serve as similarity measures; and quantum
information, where anticommuting operator sets and their ordered
products occur. In all three cases it would first have to be checked
whether the families occurring there satisfy the hypotheses. This is
open.

\subsection{Origin of the Question}\label{origin-of-the-question}

We state the origin openly, because it contributes to understanding and
because concealing it would improve nothing.

The configuration arose in the modeling of relation systems, in the
fifth volume of the series ``Winkel statt Zeit'' \citep{Jenzer2026}. There,
an arrangement with a center and \(k\) satellites bound to it is
considered; the binding to a common center is exactly the star
condition, and the theorems \(S5.10\) and \(S5.11\) there correspond to
our Theorem 4.2 and Theorem 5.2. The consequence drawn there, that
several satellites are independent for the chain but not for the loop,
is reproduced in this work as the contrast between Theorem 4.2 and
Corollary 6.2, freed from any interpretation.

This work takes over from the prior work nothing but the question.
Notation, terminology, hypotheses, proofs, and all corollaries of
Chapter 6 are newly set up here; Theorem 5.2 is proved in a stronger
version than there, and Remark 4.4, Theorem 6.8, and Chapters 6.3 to 6.7
have no counterpart in the prior work. Of the interpretations given
there, none is adopted and none is disputed; they are irrelevant to the
theorems of this work.

\subsection{On Terminology}\label{on-terminology}

A note on terminology, because it avoids misunderstandings. The prior
work calls \(\langle \Theta_k \rangle_0\) a ``holonomy.'' We avoid this
word. In differential geometry and gauge theory, holonomy denotes the
parallel transport of a connection form along a closed curve; here there
is no connection form, no basis over a space, and no transport, but a
product in an algebra. The echo would be misleading. We instead speak of
a star-ordered rotor loop and of its order defect. ``Loop'' means here
only that the word \(\Theta_k\) contains every generator with both signs
of the parameter and is closed in this purely combinatorial sense.
Should an interpretation as genuine holonomy emerge later, that would be
a separate result and not a renaming.

\begin{center}\rule{0.5\linewidth}{0.5pt}\end{center}

\pagebreak

\section{Distinction from Known Results and Open
Questions}\label{distinction-from-known-results-and-open-questions}

\begin{quote}
This chapter names what is standard, what we consider new, and where our
knowledge of the literature reaches its limits.
\end{quote}

\subsection{What Is Standard}\label{what-is-standard}

We claim none of the following. The apparatus of Clifford algebras,
reversion with its grade signs, and scalar projection are standard
\citep{Lounesto2001,Porteous1995,HestenesSobczyk1984,DoranLasenby2003,Hitzer2012}.
 Lemma 3.2 is known. The sign
\((-1)^{m(m+1)/2}\) appears in \citep[Eq.~3.4]{LundholmSvensson2009};
the structure with period four is developed in \citep[\S 3.2]{Shirokov2018} and
\citep{MarchukShirokov2008}. Lemma 3.3 is the elementary reduction rule
for blade products and is contained in every presentation. The
orthogonality of subset products with respect to the scalar form, i.e.,
the statement \(\langle B_S B_T \rangle_0 = 0\) for \(S \neq T\) under
suitable hypotheses, is almost literally contained in \citep[Theorem~14]{Shirokov2018} 
and is also used in \citep{Mandolesi2024}. We do not rely on it,
because our Lemma 3.4 obtains the weaker but sufficient statement for us
from a conjugation argument and can do without the star condition.
Whoever wishes to assume the stronger statement obtains Theorem 5.2 in
one fewer step, but under a stronger hypothesis. Theorem 4.2 for
\(k = 2\) is contained in \citep[Eq.~2.149]{DoranLasenby2003}. The
generalization is an immediate consequence; we consider it folklore. The
technique of writing a sign pattern with period four as the real part of
a product with a root of unity is classical combinatorics of lacunary
binomial sums; it goes back to Ramus and is presented in modern form in
Dobson (2017).

\subsection{What We Consider New}\label{what-we-consider-new}

To the best of our knowledge of the literature, the following points are
new. We formulate this with reservation and name the limits of our
search below.

\begin{itemize}
\tightlist
\item
  The identification of the object \(\Theta_k\), i.e., the same-order
  reversion product, as a quantity worthy of study in its own right,
  with the distinction from the trivial reverse-order version.
\item
  Theorem 5.2 with its closed subset formula and the sign law
  \(\eta(m) = (-1)^{\lfloor m/2 \rfloor}\) for arbitrary \(k\).
\item
  The compact form (Corollary 6.1) and the polar form (Corollary 6.3).
\item
  The complete zero criterion (Corollary 6.5) and the proof that
  \(k = 1\) is the only zero-free case.
\item
  The proof of non-separability for \(k \geq 2\) (Corollary 6.2).
\item
  The lacunary coefficient rule in the symmetric case (Corollary 6.7).
\item
  The finding that Theorem 5.2 holds without the star condition, but
  Theorem 4.2 does not, with the counterexample in Remark 4.4.
\item
  The hyperbolic supplementary theorem 6.8.
\item
  The permuted loop formula (Theorem 5.5) with the reduction of the
  reverse order to the limiting case of all positive signs (Corollary
  5.6). The conjecture in Remark 5.7 we do not present as a contribution
  but with explicit reservation as numerical evidence.
\item
  The generalization to mixed families with \(B_i^2 \in \{-1, +1\}\)
  through a uniform closed subset formula (Theorem 6.10).
\end{itemize}

We emphasize at the same time what this novelty is not. The proof of
Theorem 5.2 is elementary; it combines in a few lines what has long been
known: the blade square (Lemma 3.2), the symmetry of the scalar form
(Lemma 2.1), and the orthogonality of subset products, which appears in
stronger form in Shirokov (2018, Theorem 14). Our parity lemma 3.4
obtains the weaker form sufficient for us from a single conjugation
argument; the proof technique is standard, the precise statement in this
form we have not found in the literature and therefore present it as
Type 3. It would be inappropriate to claim a new method here. The
contribution lies in naming the object, computing it in closed form, and
sharply determining the scope of the hypotheses.

\subsection{Limits of Our Search}\label{limits-of-our-search}

We disclose what our literature search could not cover, so that the
readership can properly weigh the reservation in Chapter 8.2. The
conference proceedings of the ICCA and AGACSE series could not be
searched in full text, nor the complete volumes of the journal for
Clifford algebras; likewise, no systematic citation tracking via the
large mathematical reference databases was available. Several standard
works were accessible only via metadata and tables of contents, not in
full text. We have not found a prior publication of Theorems 5.2 and 6.8
in six search rounds with different formulations, but we explicitly
consider it possible that one exists.

\subsection{Open Fronts}\label{open-fronts}

\textbf{OF 1: Characterization of the Star Condition.} Every star family
satisfies the star condition (Lemma 3.5), the triangle family violates
it (Remark 4.4). For families of basis bivectors the characterization is
elementary: two basis bivectors anticommute if and only if they share
exactly one generator, and a family of two-element sets in which any two
sets share exactly one element is a star or a triangle; the star
satisfies the star condition, the triangle violates it. For basis
bivectors, therefore, star condition holds if and only if star
structure. Open is the general case of arbitrary bivectors: does every
anticommuting unit family with the star condition, up to isomorphism or
basis change, have the form \(B_i = C S_i\)? We conjecture yes, but have
not proved it.

\textbf{OF 2: Higher Grades.} The scalar projection is determined; the
projections of \(\Theta_k\) onto grades \(2, 4, \ldots\) are likewise
fixed by the family and could possess closed forms. That structure
exists is shown by the case \(k = 2\), whose full decomposition is
elementary: \[
\Theta_2 = (1 - 2 w_1 w_2) \;+\; w_2 \sin\theta_1 \, B_1 \;-\; w_1 \sin\theta_2 \, B_2 \;+\; \tfrac{1}{2} \sin\theta_1 \sin\theta_2 \, B_1 B_2.
\] The first open question is the bivector projection for general \(k\);
it would be interesting because it would provide the order defect not
just as a number, but as a rotation plane. A numerical observation
sharpens the question: in our samples the bivector part of \(\Theta_k\)
for \(k \le 3\) is always simple, and no pseudoscalar part occurs; from
\(k = 4\) a grade-4 part appears. The structural question thus arises
anew from \(k = 4\) qualitatively.

\textbf{OF 3: Powers and Longer Words.} Natural candidates are words of
the form \(P_k Q_k^n\) or \(P_k Q_k P_k Q_k^{-1}\). For powers the
situation is split: for \(k = 2\), \(\Theta_2\) is the group commutator
\(U_1 U_2 \widetilde{U_1} \widetilde{U_2}\) and, by the decomposition in
OF 2, a single-plane rotor;
\(\langle \Theta_2^n \rangle_0 = T_n(\langle \Theta_2 \rangle_0)\) holds
with the Chebyshev polynomial \(T_n\), and powers are already settled by
the scalar formula. We observe the same reduction numerically also for
\(k = 3\). From \(k = 4\), \(\langle \Theta_k^2 \rangle_0\) is no longer
determined by \(\langle \Theta_k \rangle_0\); a closed form for powers
would have to go through the principal angles of the rotor and is open,
as are the remaining longer words.

\begin{center}\rule{0.5\linewidth}{0.5pt}\end{center}

\pagebreak

\section{Appendix A:
Reproducibility}\label{appendix-a-reproducibility}

\begin{quote}
This appendix describes the accompanying verification script and names
the achieved accuracies. It replaces no proof but independently checks
every proof.
\end{quote}

\subsection{Purpose and Principle}\label{purpose-and-principle}

Every claim of this work is checkable by hand; the proofs of Chapters 3
to 6 are complete and short. Additionally we provide a verification
script that independently checks every statement, partly symbolically
and partly numerically, via three separate paths: through an own
implementation of Clifford multiplication on basis blades, through an
explicit complex matrix representation of the generators, and through an
independent symbolic evaluation of the closed formulas. Two of these
paths do not use the theorems of this work but compute the product
\(\Theta_k\) directly; the agreement is therefore not circular. The
principle is as follows: where a statement is exact, it is checked
symbolically and the difference must be identically zero. Where a
statement is only numerically accessible, it is checked over many random
parameter sets and the largest deviation encountered is reported.

\subsection{The Matrix Check}\label{the-matrix-check}

The numerical check is based on an explicit representation. For \(n\)
generators the script constructs \(n\) complex matrices of size
\(2^{\lceil n/2 \rceil}\) as Kronecker products of the Pauli matrices,
multiplied by \(\mathrm{i}\), so that each matrix has square \(-1\) and
any two of them anticommute. In this representation the scalar part of
an element equals its normalized trace, \[
\langle x \rangle_0 = \frac{1}{\dim} \operatorname{tr}(x),
\] which the script uses to compute \(\langle \Theta_k \rangle_0\). The
rotors are formed as matrices, the products \(P_k\) and \(Q_k\) are
multiplied out factor by factor, and the trace is compared with the
compact form from Corollary 6.1. The theorems of this work do not enter
this computation.

\subsection{The Checked Statements}\label{the-checked-statements}

Part A, symbolic check. For \(k = 1, \ldots, 7\), Theorem 4.2 and
Theorem 5.2 are checked with exact symbolic arithmetic on basis blades.
Additionally, the trivial identity introduced in Chapter 1.1, \[
(U_1 \cdots U_k)(\widetilde U_k \cdots \widetilde U_1) = 1,
\] is confirmed; it distinguishes the subject of this work from the
trivial version. All differences are exactly zero.

Part B, check cases. The expressions for \(k = 1, 2, 3\) given in
Chapter 6.5 are derived symbolically from the subset formula and
compared with the direct computation. Exact agreement.

Part C, matrix representation. For \(k = 1, \ldots, 10\) and three
random parameter sets each, the path from Chapter 9.2 is carried out.
Largest deviation encountered: \(4.441 \times 10^{-16}\), i.e., in the
range of double machine precision.

Part D, symmetric case. Corollary 6.6 is checked symbolically as a
polynomial identity in \(w\) for \(k = 1, \ldots, 12\). Exact agreement.

Part E, sign law. The three representations of \(\eta(m)\) from Lemma
5.1, namely \((-1)^{m(m+3)/2}\), \((-1)^{\lfloor m/2 \rfloor}\), and
\(\operatorname{Re}[(1-\mathrm{i})\mathrm{i}^m]\), are checked for
equality for \(m = 0, \ldots, 15\). Exact agreement, sign sequence
\(+,+,-,-\) with period four.

Part F, polar form, zeros, coefficients. Corollary 6.3 is checked over
2400 random parameter sets for \(k = 1, \ldots, 12\) against the compact
form; largest deviation \(1.277 \times 10^{-15}\). The identity
\(\rho_i^2 = 1 - \frac{1}{2}\sin^2\theta_i\) is symbolically confirmed,
difference exactly zero. The zero criterion from Corollary 6.5 is
checked at three explicitly constructed zeros.

\begin{longtable}{ccccc}
\toprule
$k$ & $\theta_i$ & $\sum_i \alpha_i$ & symmetric? & $\langle \Theta_k \rangle_0$ \\
\midrule
\endhead
\bottomrule
$2$ & $1.997875$ & $\frac{3\pi}{4}$ & yes & $1.1 \times 10^{-16}$ \\
$3$ & $\frac{\pi}{2}$ & $\frac{3\pi}{4}$ & yes & $3.3 \times 10^{-16}$ \\
$4$ & $2.302554$ & $\frac{7\pi}{4}$ & yes & $-4.4 \times 10^{-16}$ \\
\end{longtable}

The lacunary coefficient rule from Corollary 6.7 is checked symbolically
for \(k = 5, 9, 13\); the vanishing exponents are \(j = 1,5\),
\(1,5,9\), and \(1,5,9,13\), i.e., exactly \(j \equiv 1 \pmod 4\).

\begin{longtable}{lll}
\toprule
Checked statement & Families & Largest deviation \\
\midrule
\endhead
\bottomrule
Theorem 5.2, loop formula & all & $5.6 \times 10^{-16}$ \\
Theorem 4.2, chain formula, with common generator & & $1.1 \times 10^{-16}$ \\
Theorem 4.2, chain formula, without common generator & & $9.4 \times 10^{-1}$ \\
\end{longtable}

The loop formula thus holds without the star condition, the chain
formula does not. This confirms Theorem 5.2 in its general version and
Remark 4.4 in its function as a counterexample.

Part G, scope of hypotheses. The script enumerates in
\(\mathrm{Cl}_{0,n}\) for \(n = 3, \ldots, 7\) all families of basis
bivectors that anticommute pairwise, and checks both theorems for each
of them with ten random parameter sets each. The result sharply
separates the two theorems.

Part H, counterexample. For the triangle family in \(\mathrm{Cl}_{0,3}\)
it is confirmed: \(B_i^2 = -1\) for all three, pairwise anticommutation,
no common generator, and \(\langle B_1 B_2 B_3 \rangle_0 = -1\).

Part I, hyperbolic case. Theorem 6.8 is checked for \(k = 1, \ldots, 8\)
with 30 random parameter sets each in a representation with
\(B_i^2 = +1\); largest deviation \(9.992 \times 10^{-16}\).
Additionally,
\(\eta^+(m) = \operatorname{Re}[(1+\mathrm{i})\mathrm{i}^m]\) is exactly
confirmed for \(m = 0, \ldots, 11\).

Part J, permuted loop. Theorem 5.5 is checked for \(k = 1, \ldots, 6\)
on star families with thirty random parameter sets each and randomly
drawn permutations. The product \(\Theta_k\) is multiplied out factor by
factor and its scalar part is compared with the permuted loop formula;
the theorems of this work do not enter this computation. Additionally,
Corollary 5.6 is symbolically confirmed for the reverse permutation
\(\tau_{\mathrm{rev}}\). All differences are exactly zero or, in
numerical checks, in the range of double machine precision.

Part K, mixed families. Theorem 6.10 is checked for \(k = 1, \ldots, 8\)
with thirty random parameter sets each and randomly chosen types
\(\varepsilon_i \in \{-1, +1\}\). For \(\varepsilon_i = -1\)
trigonometric rotors are used, for \(\varepsilon_i = +1\) hyperbolic
rotors. The product \(\Theta_k\) is computed directly in a Clifford or
matrix representation and its scalar part is compared with the subset
formula from Theorem 6.10. The largest deviation encountered is
reported.

\subsection{Availability}\label{availability}

The script is a single file in Python 3 and requires no dependencies
beyond libraries for symbolic and numerical arithmetic. It runs without
arguments, outputs each check step with its result, and aborts with an
error message upon any violation of a checked statement. The run ends
with the line \emph{ALL CHECKS PASSED}. Manuscript and script are
deposited at the address
https://codeberg.org/Jenziner/star-ordered-rotor-loops and under the DOI
10.5281/zenodo.22159164, both under CC BY 4.0. We explicitly invite not
merely running the script but re-implementing the checked statements
independently; the proofs are short enough that this is feasible with
reasonable effort.

\begin{center}\rule{0.5\linewidth}{0.5pt}\end{center}

\pagebreak

\section{Lists}\label{lists}

\subsection{Glossary}\label{glossary}

\begin{itemize}
\tightlist
\item
  \textbf{Anticommuting unit family.} A finite family
  \(B_1, \ldots, B_k\) of elements of a real Clifford algebra with
  \(B_i B_j = -B_j B_i\) for \(i \neq j\) and uniform square
  \(B_i^2 = -1\) (type \(-1\)) or \(B_i^2 = +1\) (type \(+1\)).
\item
  \textbf{Leaf generator.} The generators \(S_1, \ldots, S_k\) of a star
  family, each assigned to exactly one \(B_i\).
\item
  \textbf{Triangle family.} The family \(e_1 e_2, e_1 e_3, e_2 e_3\) in
  \(\mathrm{Cl}_{0,3}\); an anticommuting unit family without a common
  generator and without the star condition.
\item
  \textbf{Mixed family.} An anticommuting unit family whose members
  carry different squares, \(B_i^2 = \varepsilon_i \in \{-1, +1\}\);
  subject of Theorem 6.10.
\item
  \textbf{Same-order reversion product.} The product
  \(Q_k = \widetilde U_1 \widetilde U_2 \cdots \widetilde U_k\), in
  which the indices have the same order as in \(P_k\). It is to be
  distinguished from the reversion
  \(\widetilde P_k = \widetilde U_k \cdots \widetilde U_1\).
\item
  \textbf{Inversion count.} The number \(\iota(\tau, S)\) of pairs in a
  subset \(S\) whose ascending order is reversed under a permutation
  \(\tau\) (Definition D5.2).
\item
  \textbf{Compact form.} The representation of the order defect as the
  real part of a product of complex numbers (Corollary 6.1).
\item
  \textbf{Lacunary coefficient rule.} The statement that in the
  symmetric case every fourth coefficient of the polynomial in \(w\)
  vanishes (Corollary 6.7).
\item
  \textbf{Order defect.} The scalar part \(\langle \Theta_k \rangle_0\);
  it measures how far the same-order reversion product deviates from
  canceling the forward product.
\item
  \textbf{Permuted loop.} The word
  \(\Theta_k(\tau) = P_k Q_k^{(\tau)}\), in which the reversion product
  carries an index order determined by the permutation (Definition D5.2
  and Theorem 5.5).
\item
  \textbf{Polar form.} The representation of the order defect as a
  product of real magnitudes with a single cosine of an angle sum
  (Corollary 6.3).
\item
  \textbf{Scalar projection.} The linear map \(\langle \cdot \rangle_0\)
  that assigns to an element its grade-zero component.
\item
  \textbf{Star condition.} The requirement \(\langle B_S \rangle_0 = 0\)
  for every non-empty subset \(S\). It is needed for Theorem 4.2 and not
  for Theorem 5.2.
\item
  \textbf{Star family.} A family \(B_i = C S_i\) from a central
  generator \(C\) and \(k\) leaf generators \(S_i\), all pairwise
  anticommuting with square \(-1\).
\item
  \textbf{Star-ordered rotor loop.} The word \(\Theta_k = P_k Q_k\) from
  the ordered forward product and the same-order reversion product.
\item
  \textbf{Subset formula.} The representation of the order defect as a
  sum over all subsets of the index set (Theorem 5.2).
\item
  \textbf{Subset product.} The product \(B_S = B_{i_1} \cdots B_{i_m}\)
  over a subset \(S = \{i_1 < \cdots < i_m\}\) in ascending index order,
  with \(B_\emptyset = 1\).
\item
  \textbf{Period-four sign law.} The sequence
  \(\eta(m) = (-1)^{\lfloor m/2 \rfloor}\) with the pattern \(+,+,-,-\)
  (Lemma 5.1).
\item
  \textbf{Central generator.} The generator \(C\) occurring in all
  \(B_i\) of a star family. The word ``central'' refers to the star
  arrangement, not to the center of the algebra; \(C\) is not central in
  the algebraic sense.
\end{itemize}

\pagebreak

\subsection{Notation}\label{notation}

\begin{longtable}{ll}
\toprule
Symbol & Meaning \\
\midrule
\endhead
\bottomrule
\(\mathcal{C}\) & real Clifford algebra in which the computation takes
place \\
\(\mathrm{Cl}_{p,q}\) & real Clifford algebra with \(p\) generators of
square \(+1\) and \(q\) of square \(-1\) \\
\(C\) & central generator, \(C^2 = -1\) \\
\(S_i\) & leaf generator, \(S_i^2 = -1\) \\
\(B_i\) & unit bivector, \(B_i = C S_i\), \(B_i^2 = -1\) \\
\(k\) & number of bivectors or leaves \\
\(\theta_i\) & real parameter of the \(i\)-th rotor \\
\(\chi_i\) & real parameter in the hyperbolic case (Theorem 6.8) \\
\(c_i, s_i\) & shorthand for \(\cos\frac{\theta_i}{2}\) and
\(\sin\frac{\theta_i}{2}\) \\
\(w_i\) & weight \(\sin^2\frac{\theta_i}{2} \in [0,1]\) \\
\(U_i\) & rotor \(c_i + s_i B_i\) \\
\(\widetilde{x}\) & reversion of \(x\) \\
\(\widetilde{U}_i\) & \(c_i - s_i B_i\) \\
\(\langle x \rangle_0\) & scalar projection of \(x\) \\
\(P_k\) & ordered forward product \(U_1 U_2 \cdots U_k\) \\
\(Q_k\) & same-order reversion product
\(\widetilde U_1 \cdots \widetilde U_k\) \\
\(\Theta_k\) & star-ordered rotor loop \(P_k Q_k\) \\
\(S, T, D\) & subsets of \(\{1, \ldots, k\}\) \\
\(S \triangle T\) & symmetric difference \\
\(|S|\) & cardinality of \(S\) \\
\(B_S\) & subset product, \(B_\emptyset = 1\) \\
\(a_S\) & coefficient \(\prod_{i \in S} s_i \prod_{i \notin S} c_i\) \\
\(\sigma(S,T)\) & sign of the blade reduction, Lemma 3.3 \\
\(\tau\) & permutation of the index order in the reversion product,
Definition D5.2 \\
\(\tau_{\mathrm{rev}}\) & reverse permutation,
\(\tau_{\mathrm{rev}}(j) = k+1-j\) \\
\(\iota(\tau, S)\) & inversion count of \(\tau\) on \(S\), Definition
D5.2 \\
\(Q_k^{(\tau)}\) & permuted reversion product
\(\widetilde U_{\tau(1)} \cdots \widetilde U_{\tau(k)}\) \\
\(\Theta_k(\tau)\) & permuted loop \(P_k Q_k^{(\tau)}\) \\
\(\varepsilon_i\) & sign type of the \(i\)-th family member,
\(B_i^2 = \varepsilon_i \in \{-1, +1\}\), Theorem 6.10 \\
\(\eta(m)\) & sign law \((-1)^{\lfloor m/2 \rfloor}\) \\
\(\eta^{+}(m)\) & sign law in the hyperbolic case,
\((-1)^{m(m+1)/2}\) \\
\(\rho_i\) & magnitude
\(\sqrt{\cos^4\frac{\theta_i}{2} + \sin^4\frac{\theta_i}{2}}\) \\
\(\alpha_i\) & angle
\(\arctan(\tan^2\frac{\theta_i}{2}) \in [0, \frac{\pi}{2})\) \\
\(\mathrm{i}\) & imaginary unit in \(\mathbb{C}\), not an element of
\(\mathcal{C}\) \\
\(\operatorname{Re}\) & real part of a complex number \\
\(\blacksquare\) & end of a proof \\
\end{longtable}

\pagebreak

\bibliographystyle{plainnat}
\bibliography{star-ordered-rotor-loops}

@article{Curtright2014,
  author = {Curtright, T. L. and Fairlie, D. B. and Zachos, C. K.},
  doi = {10.3842/SIGMA.2014.084},
  journal = {Symmetry, Integrability and Geometry: Methods and Applications},
  pages = {084},
  title = {A Compact Formula for Rotations as Spin Matrix Polynomials},
  volume = {10},
  year = {2014}
}

@article{Dobson2017,
  archiveprefix = {arXiv},
  author = {Dobson, J. B.},
  eprint = {1610.09361},
  journal = {International Journal of Mathematics and Computer Science},
  pages = {27--42},
  title = {A Matrix Variation on {R}amus's Identity for Lacunary Sums of Binomial Coefficients},
  volume = {12},
  year = {2017}
}

@book{DoranLasenby2003,
  author = {Doran, Chris and Lasenby, Anthony},
  doi = {10.1017/CBO9780511807497},
  publisher = {Cambridge University Press},
  title = {Geometric Algebra for Physicists},
  year = {2003}
}

@incollection{Eckmann1994,
  author = {Eckmann, Beno},
  booktitle = {The Hilton Symposium 1993},
  doi = {10.1090/crmp/006-02},
  pages = {23--35},
  publisher = {American Mathematical Society},
  series = {CRM Proceedings and Lecture Notes},
  title = {Hurwitz-{R}adon Matrices Revisited: From Effective Solution of the {H}urwitz Matrix Equations to {B}ott Periodicity},
  volume = {6},
  year = {1994}
}

@book{HestenesSobczyk1984,
  author = {Hestenes, David and Sobczyk, Garret},
  doi = {10.1007/978-94-009-6292-7},
  publisher = {Springer},
  series = {Fundamental Theories of Physics},
  title = {Clifford Algebra to Geometric Calculus: A Unified Language for Mathematics and Physics},
  volume = {5},
  year = {1984}
}

@article{Hitzer2012,
  archiveprefix = {arXiv},
  author = {Hitzer, Eckhard},
  eprint = {1306.1660},
  journal = {Journal of the Society of Instrument and Control Engineers},
  number = {4},
  pages = {338--350},
  title = {Introduction to {C}lifford's Geometric Algebra},
  volume = {51},
  year = {2012}
}

@article{Hrubes2016,
  author = {Hrube{\v{s}}, Pavel},
  doi = {10.1016/j.laa.2015.12.015},
  journal = {Linear Algebra and Its Applications},
  pages = {494--507},
  title = {On Families of Anticommuting Matrices},
  volume = {493},
  year = {2016}
}

@misc{Jenzer2026,
  author = {Jenzer, Nik},
  doi = {10.5281/zenodo.20769239},
  publisher = {Zenodo},
  title = {Winkel statt {Z}eit -- {B}and 5 -- Die Relationsalgebra},
  year = {2026}
}

@book{Lounesto2001,
  author = {Lounesto, Pertti},
  doi = {10.1017/CBO9780511526022},
  edition = {2nd},
  publisher = {Cambridge University Press},
  series = {London Mathematical Society Lecture Note Series},
  title = {Clifford Algebras and Spinors},
  volume = {286},
  year = {2001}
}

@misc{LundholmSvensson2009,
  archiveprefix = {arXiv},
  author = {Lundholm, Douglas and Svensson, Lars},
  eprint = {0907.5356},
  note = {Lecture notes},
  title = {Clifford Algebra, Geometric Algebra, and Applications},
  year = {2009}
}

@article{Mandolesi2024,
  author = {Mandolesi, Andr{\'e} L. G.},
  doi = {10.1007/s00006-024-01357-4},
  journal = {Advances in Applied Clifford Algebras},
  number = {5},
  pages = {Article 56},
  title = {Multivector Contractions Revisited, {P}art {I}},
  volume = {34},
  year = {2024}
}

@article{MarchukShirokov2008,
  author = {Marchuk, Nikolai G. and Shirokov, Dmitry S.},
  doi = {10.1007/s00006-008-0066-y},
  journal = {Advances in Applied Clifford Algebras},
  pages = {237--254},
  title = {Unitary Spaces on {C}lifford Algebras},
  volume = {18},
  year = {2008}
}

@book{Porteous1995,
  author = {Porteous, Ian R.},
  doi = {10.1017/CBO9780511470912},
  publisher = {Cambridge University Press},
  title = {Clifford Algebras and the Classical Groups},
  year = {1995}
}

@misc{SarkarVanDenBerg2019,
  archiveprefix = {arXiv},
  author = {Sarkar, Ritajit and van den Berg, Ewout},
  doi = {10.48550/arXiv.1909.08123},
  eprint = {1909.08123},
  title = {On Sets of Commuting and Anticommuting {P}aulis},
  year = {2019}
}

@incollection{Shirokov2018,
  author = {Shirokov, Dmitry},
  booktitle = {Geometry, Integrability and Quantization},
  doi = {10.7546/giq-19-2018-11-53},
  editor = {Mladenov, Ivail{\`o} M. and Yoshioka, Akira},
  pages = {11--53},
  publisher = {Avangard Prima},
  title = {Clifford Algebras and Their Applications to {L}ie Groups and Spinors},
  volume = {19},
  year = {2018}
}

@article{Shirokov2023,
  author = {Shirokov, Dmitry},
  doi = {10.3390/math11163607},
  journal = {Mathematics},
  number = {16},
  pages = {3607},
  title = {Development of the Method of Averaging in {C}lifford Geometric Algebras},
  volume = {11},
  year = {2023}
}

\end{document}